\documentclass[10pt]{amsart}
\usepackage[
hmarginratio={1:1},     % equal left and right margins
vmarginratio={1:1},     % equal top and bottom margins
textwidth=400pt,        % new text width
textheight=50.5pc,
heightrounded,          % always useful
]{geometry}
\usepackage{mathrsfs}
\usepackage{epsfig}
\usepackage[dvipsnames]{xcolor}
\usepackage{graphicx}
\usepackage{amsmath}
\usepackage{amsfonts}
\usepackage{amssymb}
\usepackage{bbm}
\usepackage{bm}
\usepackage{amsthm}
\usepackage{amscd}
\usepackage{verbatim}
\usepackage{tikz}
\usepackage{lmodern}
\usepackage[T1]{fontenc} 
\usepackage[autostyle]{csquotes}
\usepackage{ytableau}
\usepackage{enumitem}
\usetikzlibrary{arrows}
\usepackage{multicol}

\usepackage{aliascnt}
\usepackage[hidelinks]{hyperref}
\usepackage[nameinlink,noabbrev]{cleveref}

\tikzset{
vertex/.style={circle,draw,minimum size=1.5em},
edge/.style={->,> = latex'}
}

\usepackage{mathtools}

\theoremstyle{plain}
\newtheorem{theorem}{Theorem}[section]

\newaliascnt{proposition}{theorem}
\newtheorem{proposition}[proposition]{Proposition}
\aliascntresetthe{proposition}

\newaliascnt{corollary}{theorem}
\newtheorem{corollary}[corollary]{Corollary}
\aliascntresetthe{corollary}

\newaliascnt{lemma}{theorem}
\newtheorem{lemma}[lemma]{Lemma}
\aliascntresetthe{lemma}

\theoremstyle{definition}

\newaliascnt{definition}{theorem}

\aliascntresetthe{definition}

\newaliascnt{remark}{theorem}
\newtheorem{remark}[remark]{Remark}
\aliascntresetthe{remark}

\newaliascnt{example}{theorem}

\aliascntresetthe{example}

\numberwithin{equation}{section}

\DeclareMathAlphabet{\mathpzc}{OT1}{pzc}{m}{it}

\definecolor{MyBlue1}{RGB}{229, 119, 167}
\definecolor{MyBlue2}{RGB}{61, 101, 165}
\definecolor{MyBlue3}{RGB}{124, 161, 204}

\newcommand{\norm}[1]{\| #1 \|}

\usepackage{marginnote,ragged2e}
\let\oldenumerate=\enumerate
\def\enumerate{
\oldenumerate
\setlength{\itemsep}{5pt}
}
\let\olditemize=\itemize
\def\itemize{
\olditemize
\setlength{\itemsep}{5pt}
}

\setlist[enumerate]{leftmargin=2.5em}

\begin{document}

\title{A general framework for inequalities on simple graphs}%[Bounds for the spectral radius and energy of graphs via Schur-convexity of random vector norms]{Bounds for the spectral radius and energy of graphs via Schur-convexity of random vector norms}
\author[Bouthat]{Ludovick Bouthat}
\author[Ch\'avez]{\'Angel Ch\'avez}
\author[Sheng]{Sam Sheng}

\address{Département de mathématiques et statistique, Université Laval, 2325 Rue de l'Université, Québec, QC G1V 0A6} 
\email{ludovick.bouthat.1@ulaval.ca}

\address{Department of Mathematics and Computer Science, Davidson College, 209 Ridge Rd, Box 5000, Davidson, NC 28035} 
\email{anchavez@davidson.edu}

\address{Department of Economics, Columbia University, 1022 International Affairs Building, Mail Code 3308, 420 West 118th Street, New York, NY 10027} 
\email{zs2819@columbia.edu}

%\begin{abstract}
%The \emph{random vector norms} of a graph are the absolute moments of sums of iid random variables weighted by the eigenvalues of the graph. Schur-convexity yields sharp bounds for the random vector norms of a graph in terms of its spectral radius and graph energy. Random vector norms are induced by an underlying probability distribution, and changing the underlying probability distribution yields different applications. For instance, an exponential distribution yields norms that can distinguish between noncospectral singularly cospectral graphs. A uniform distribution allows us to establish  lower bounds for the spectral radius of triangle- and square-free graphs. Finally, a Bernoulli distribution allows us to establish bounds for the spectral radius and energy of a graph in terms of its number of vertices, edges, triangles and squares. All inequalities we establish are sharp.
%\end{abstract}

\begin{abstract}
A general framework is developed for deriving sharp inequalities on simple graphs from majorization and Schur-convexity. After establishing majorization relations between the spectrum of an arbitrary graph and the spectra of the complete, complete bipartite, and matching graphs, it is shown that every positive Schur-convex spectral functional yields several sharp inequalities relating $\lambda_1$, $|\lambda_n|$, and $\|G\|_\ast$. This reduces the problem of proving graph inequalities to the choice of a suitable Schur-convex function. This optimization problem is then studied within the family of random vector norms, whose moment and cumulant expansions connect the framework to the numbers of closed walks. This yields new sharp results, recovers classical inequalities from a unified viewpoint, and produces further bounds in settings such as triangle-free and square-free graphs.
\end{abstract}

\maketitle

\section{Introduction}

\subsection{Foreword}

The spectrum of a graph encodes a remarkable amount of structural information, from connectivity to extremal behavior. It is therefore natural to study graph parameters that summarize spectral information in a single quantity. Among the most classical examples are norms built from the adjacency matrix, which have played a central role in spectral graph theory for several decades.

Throughout this paper, all graphs are assumed to be simple. If $G$ is a graph of order $n$, its \emph{singular values} are defined as the singular values of its adjacency matrix $A(G)$, which we denote by
\[
\sigma_1 \geq \sigma_2 \geq \cdots \geq \sigma_n.
\]
Since $A(G)$ is always real and symmetric, it is a normal matrix, and its singular values coincide with the absolute values of its eigenvalues.

Several classical norms of graphs are defined in terms of this singular spectrum. For instance, the \emph{energy} of $G$ (also known as the \emph{nuclear norm} or \emph{trace norm}) is the quantity
\[
\norm{G}_{\ast} = \sigma_1 + \sigma_2 + \cdots + \sigma_n,
\]
introduced by Gutman in 1978 \cite{Gutman1}. The \emph{spectral norm} (or spectral radius) is $\norm{G} = \sigma_1$. More generally, one may consider the Schatten $p$-norms for $p\geq 1$:
\[
\norm{G}_{S_p} = \left(\sigma_1^p + \sigma_2^p + \cdots + \sigma_n^p\right)^{1/p}.
\]
In particular, $\norm{G}_{S_1} = \norm{G}_{\ast}$ and $\norm{G}_{S_\infty} = \norm{G}$.

There are hundreds of articles devoted to the study of $\norm{G}_{\ast}$ and $\norm{G}$. Three decades of progress on the energy of graphs is summarized, for instance, by Li, Shi and Gutman in \cite{GutmanLiShi}. The lasting interest in $\norm{G}_{\ast}$ and $\norm{G}$ has recently expanded to include the study of more obscure norms of graphs. In particular, Nikiforov establishes several results for the Ky Fan and Schatten norms of graphs in \cite{Nikiforov2,Nikiforov1}.

While these classical norms are natural and important, they are usually studied individually, and the proofs of the corresponding inequalities are often specific to the considered functional. The main purpose of this paper is to show that many such inequalities can be derived from a common principle, and that this principle may lead to new sharp inequalities.

The language underlying this principle is that of majorization and Schur-convexity. If $\boldsymbol{x},\boldsymbol{y}\in\mathbb{R}^n$, we denote by $\boldsymbol{x}^{\downarrow}=(x_1^{\downarrow},x_2^{\downarrow},\dots,x_n^{\downarrow})$ the decreasing rearrangement of $\boldsymbol{x}$. We say that $\boldsymbol{x}$ is \emph{majorized} by $\boldsymbol{y}$, and write $\boldsymbol{x}\prec\boldsymbol{y}$, if
\[
\sum_{j=1}^k x_j^{\downarrow}\leq \sum_{j=1}^k y_j^{\downarrow}
\qquad (k=1,2,\dots,n),
\]
with equality for $k=n$. A function $f:\mathbb{R}^n\to\mathbb{R}$ is then called \emph{Schur-convex} if $f(\boldsymbol{x})\leq f(\boldsymbol{y})$ whenever $\boldsymbol{x}\prec\boldsymbol{y}$.

Our starting point is a collection of majorization relations comparing the spectrum of an arbitrary simple graph with the spectra of three canonical extremal families: the complete graphs, the complete bipartite graphs, and the matching graphs. Combined with Schur-convexity, these comparisons yield our main result, \Cref{thm - main}, which may be viewed as a general mechanism for producing sharp inequalities on graphs. More precisely, once a positive spectral functional is known to be Schur-convex in the eigenvalues, \Cref{thm - main} yields a chain of inequalities relating this functional to $\lambda_1$, $|\lambda_n|$, and $\norm{G}_{\ast}$, with equality attained by natural extremal graphs.

This shifts the emphasis of the problem. Rather than proving one inequality at a time, we reduce the task to the choice of a suitable Schur-convex function. The rest of the paper is organized around this question: which functions lead to the sharpest instances of \Cref{thm - main}? In this sense, the main objective of the paper is not merely to prove a collection of isolated inequalities, but to provide a systematic method for generating many of them from the same underlying principle.

\subsection{Random vector norms of graphs}\label{sec - random vector norms}

A particularly useful family of candidates is provided by the \emph{random vector norms}, which are norms on the space $\mathrm{H}_n$ of $n\times n$ Hermitian matrices introduced by Chávez, Garcia and Hurley in \cite{survey,Chavez1,Chavez2}. If $A \in H_n$ has eigenvalues $\lambda_1 \geq \lambda_2 \geq \cdots \geq \lambda_n$, $d \geq 1$, and $X_1,\dots,X_n$ are iid copies of a nonconstant random variable $X$ with $\mathbb{E}|X|^d<\infty$, then
\[
\|A\|_{X,d}
=
\Gamma(d+1)^{-1/d}
\left(
\mathbb{E}\left|\lambda_1X_1+\cdots+\lambda_nX_n\right|^d
\right)^{1/d}
\]
surprisingly defines a norm on $H_n$ \cite[Theorem 3 (a)]{Chavez2} (which is also submultiplicative under mild conditions; see \cite{Bouthat}). Since the adjacency matrix of a simple graph is Hermitian, this gives rise to a family of graph norms $\|G\|_{X,d}$ indexed by a distribution and an order parameter $d$.

These norms fit naturally into our framework for two reasons. First, they are Schur-convex in the eigenvalues, and therefore may be inserted directly into \Cref{thm - main}. That is, if $\boldsymbol{\lambda}(G) :=(\lambda_1, \lambda_2, \ldots, \lambda_n)$ denotes the vector of eigenvalues of a graph $G$, then the function $\boldsymbol{\lambda}(G)\mapsto \norm{G}_{X, d}$ is Schur-convex for every even $d\geq 2$ \cite{Chavez1}.

Second, when $d$ is even and the \emph{moment generating function} of $X$ exists, then Ch\'avez, Garcia and Hurley showed in \cite[Theorem 1 (b)]{Chavez1} that they admit explicit expansions that connect them to classical graph invariants. Recall that the moment generating function of a random variable $X$ is defined as the function $M_X(t):=\mathbb{E}[e^{tX}]=\sum_{k=0}^{\infty} \mathbb{E}[X^k] \frac{t^k}{k!}$. In particular, the authors proved that if $G$ is a graph of order $n$, then
\begin{equation}
	\norm{G}_{X,d}^d=[t^d] \prod_{k=1}^n M_X(\lambda_k t), \label{thm:Moments}
\end{equation}
where, here and throughout the paper, $[t^d]f(t)$ denotes the $d$th Taylor coefficient of $f(t)$. 

\smallskip

Moreover, they further proved that we may express the random vector norms as a combination of the number of closed walks of length $k$ in the graph for $k\leq d$. Henceforth, we let $c_k(G)$ denote the number of closed walks of length $k$ in $G$. Moreover, we write $c_k=c_k(G)$ when $G$ is understood. It is a fundamental fact that the number $c_k(G)$ is encoded in the spectrum of $G$ through the relation 
\begin{align*}
	c_k(G) = \lambda_1^k+\lambda_2^k+\cdots+\lambda_n^k = \operatorname{tr}\bigl(A^k(G)\bigr),
\end{align*}
where $A(G)$ is the adjacency matrix of the graph $G$.

To state this result, we first need to establish some notations and recall some definitions. A \emph{partition} of a positive integer $d$ is a tuple $\boldsymbol{\pi}=(\pi_1, \pi_2, \ldots, \pi_{\ell})$ of positive integers satisfying $\pi_1 \geq \pi_2 \geq \cdots \geq \pi_{\ell}$ and $\pi_1 + \pi_2 + \cdots + \pi_{\ell} = d$. We adopt the standard notation $\boldsymbol{\pi} \vdash d$ to indicate $\boldsymbol{\pi}$ is a partition of $d$. Given a partition $\boldsymbol{\pi}$, we define $y_{\boldsymbol{\pi}} := \prod_{i\geq 1} (i!)^{m_i} m_i!$, in which $m_i=m_i(\boldsymbol{\pi})$ denotes the multiplicity of $i$ in $\boldsymbol{\pi}$. We write $i^{m_i}$ when $i$ appears $m_i$ times in $\boldsymbol{\pi}$. For instance, $(1, 1,1,2,2)=(1^3,2^2)$. \Cref{ex:yValues} gives a few important values of $y_{\boldsymbol{\pi}}$. 

\begin{table}[ht]
	\centering
	\begin{tabular}{ |c|c|c|c|c|c|c|c|} 
		\hline
		&&&&&&&\\[-10.5pt]
		$\boldsymbol{\pi}$ & $(2)$ & $(4)$ & $(2^2)$ & $(6)$ & $(4,2)$ & $(3^2)$ & $(2^3)$\\
		\hline&&&&&&&\\[-10.5pt]
		$y_{\boldsymbol{\pi}}$ & $2!$ & $4!$ & $(2!)^3$ & $6!$ & $2!4!$ & $2! (3!)^2$ & $(2!)^33!$ \\
		\hline
	\end{tabular} 
	\vspace{6pt}
	\caption{Values of $y_{\boldsymbol{\pi}}$ for partitions of $2$, $4$ and $6$ without 1.}
	\label{ex:yValues}
\end{table}

Lastly, if $X$ admits the moment generating function $M_X(t)$, then the \emph{cumulant generating function of $X$} is defined by
\begin{align*}
	K_X(t):=\log M_X(t)=\sum_{j=1}^{\infty} \kappa_j \frac{t^j}{j!},
\end{align*} 
in which $\kappa_j$ is called the $j$th \emph{cumulant} of $X$. Given a partition $\boldsymbol{\pi}$ of $d$, we write $\kappa_{\boldsymbol{\pi}}:=\kappa_{\pi_1}\kappa_{\pi_2}\cdots \kappa_{\pi_{\ell}}$. Similarly, we also write $c_{\boldsymbol{\pi}}:=c_{\pi_1}c_{\pi_2}\cdots c_{\pi_{\ell}}$

Using these standard definitions, the authors of \cite{Chavez1} proved that if $d\geq 2$ is even, $X$ is a random variable with cumulant sequence $\kappa_1, \kappa_2, \dots$, and $G$ is a graph of order $n$ with $c_k$ closed walks of length $k$, then
\begin{equation}\label{thm:Cumulants}
	\norm{G}_{X,d}^d=\sum_{\boldsymbol{\pi}\vdash d} \frac{\kappa_{\boldsymbol{\pi}} c_{\boldsymbol{\pi}}}{ y_{\boldsymbol{\pi}}}.
\end{equation} 
We may, without loss of generality, consider the summation over all partitions without 1 appearing, since $c_1(\lambda_1,\ldots, \lambda_n) = \lambda_1+\cdots+\lambda_n = \operatorname{tr}\bigl(A(G)\bigr) = 0$ for simple graphs. 
Both the formulations \eqref{thm:Moments} and \eqref{thm:Cumulants} shall be useful to establish our results in this paper.

\subsection{Organization of the paper}

\Cref{sec:Bounds} develops the general framework. We first establish the majorization relations between the spectrum of an arbitrary graph and the spectra of the graphs $K_n$, $K_{r,s}$, and $M_n$, and then prove \Cref{thm - main}, which transforms any positive Schur-convex spectral functional into a family of sharp inequalities. We also illustrate the method with the Schatten norms, recovering natural inequalities and reformulating them in terms of closed walks.

\Cref{sec - d small} turns to the optimization problem inside the class of random vector norms. We determine the sharpest choices in small even degrees, obtain complete results for $d=2$ and $d=4$, and provide evidence for the role of the Rademacher and Schatten norms in higher degrees. This leads to sharp inequalities involving the numbers of closed walks and the extremal eigenvalues.

\Cref{sec - other} studies further distributions whose special algebraic structure yields additional applications of the framework. In particular, these choices lead to convenient or sharp bounds for spectral quantities and closed walks in settings such as triangle-free and square-free graphs.

%%%%%%%%%%%%%%%%%%%%%%%%%%%%%%%%%%%%%%%%%%%%%%%%%%%%%%%%%%%%%%%%%%%%%%%%%%%
%%%%%%%%%%%%%%%%%%%%%%%%%%%%%%%%%%%%%%%%%%%%%%%%%%%%%%%%%%%%%%%%%%%%%%%%%%%
%%%%%%%%%%%%%%%%%%%%%%%%%%%%%%%%%%%%%%%%%%%%%%%%%%%%%%%%%%%%%%%%%%%%%%%%%%%
%%%%%%%%%%%%%%%%%%%%%%%%%%%%%%%%%%%%%%%%%%%%%%%%%%%%%%%%%%%%%%%%%%%%%%%%%%%
%%%%%%%%%%%%%%%%%%%%%%%%%%%%%%%%%%%%%%%%%%%%%%%%%%%%%%%%%%%%%%%%%%%%%%%%%%%

\section{Bounds for the extreme eigenvalues and the energy}\label{sec:Bounds}
Our first goal is to establish some majorization-based tools which allow us to obtain sharp bounds for $\lambda_1$, $|\lambda_n|$ and $\norm{G}_{\ast}$ in terms of the graph spectrum. Henceforth, we let $K_n, K_{r,s}$, $P_n$ and $S_n$ denote the \emph{complete}, \emph{complete bipartite}, \emph{path} and \emph{star} graphs, respectively. Moreover, we denote by $M_n$ the \emph{matching} graph on $n$ vertices, which corresponds to a \emph{perfect matching} when $n$ is even, and to a \emph{near-perfect matching} if $n$ is odd (i.e., a perfect matching with one isolated vertex). 

%%%%%%%%%%%%%%%%%%%%%%%%%%%%%%%%%%%%%%%%%%%%%%%%%%%%%%%%%%%%%%%%%%%%%%%%%%%
%%%%%%%%%%%%%%%%%%%%%%%%%%%%%%%%%%%%%%%%%%%%%%%%%%%%%%%%%%%%%%%%%%%%%%%%%%%
%%%%%%%%%%%%%%%%%%%%%%%%%%%%%%%%%%%%%%%%%%%%%%%%%%%%%%%%%%%%%%%%%%%%%%%%%%%

\subsection{Bounds over all graphs}\label{sec:SimpleBounds}
We define $\mbox{spec}(G)$ to be the multiset of eigenvalues of a graph $G$. To establish our main result, we first begin by presenting the following spectra of classical graphs, which are well-known.

\begin{proposition}\cite[Section 1.4]{Brouwer}\label{prop:Spectra} The spectra for $P_n$, $K_n$, $K_{r,s}$ and $M_n$ are as follows:
	\begin{enumerate}
		\item[(a)] $\operatorname{spec}(P_n)= \bigl\{2\cos\bigl( \frac{k\pi}{n+1}\bigr)\,:\,k=1,2,\ldots, n\bigr\},$
		\item[(b)] $\operatorname{spec}(K_n)=\bigl\{n-1,\,-1,\,\dots,-1\bigr\}$,%\big\{n-1,\, (-1)^{[n-1]}\big\}$,
		\item[(c)] $\operatorname{spec}(K_{r,s})=\bigl\{\sqrt{rs},\,0,\, \dots,\, 0, -\sqrt{rs}\bigr\}$,%\big\{\sqrt{rs},\, 0^{[r+s-2]},\, -\sqrt{rs}\big\}$.
		
		\item[(d)] $\operatorname{spec}(M_n)\!=\!\bigl\{1,\dots,1,-1,\dots, -1\bigr\}$ if $n$ is even and $\bigl\{1,\dots,1,0,-1,\dots, -1\bigr\}$ if $n$ is odd.
	\end{enumerate}
\end{proposition}

Let us now prove a series of three key lemmas.

\begin{lemma}\label{lemma:Complete}
	Let $G$ be a simple graph of order $n$ with eigenvalues $\bm{\lambda}=(\lambda_1,\lambda_{2},\dots,\lambda_n)$ and let $\bm{\mu} = (n-1,-1,\dots,-1)$. Then,
	\[
	\frac{\lambda_1}{n-1}\,\bm{\mu} \,\prec\, \bm{\lambda} \,\prec\, |\lambda_n| \bm{\mu}.
	\]
\end{lemma}
\begin{proof}
	Let $\alpha,\beta>0$. We have $\bm{\lambda} \prec \alpha \bm{\mu}$ if and only if $\sum_{j=1}^k \lambda_j \leq \alpha (n-k)$ for all $k=1,2,\dots,n-1$. Since $\lambda_1+\cdots+\lambda_n =\operatorname{tr}(A)=0$, this is equivalent to 
	\[
	\frac{-\lambda_n-\cdots-\lambda_{k+1}}{n-k} \leq \alpha, \qquad (k=1,2,\dots,n-1).
	\]
	Similarly, we have $\beta \bm{\mu} \prec \bm{\lambda}$ if and only if
	\[
	\frac{-\lambda_n-\cdots-\lambda_{k+1}}{n-k} \geq \beta, \qquad (k=1,2,\dots,n-1).
	\]
	Now, the left-hand side is the average of the $n-k$ largest values of $-\bm{\lambda}$. Therefore, this quantity is nondecreasing, which implies that we have $\bm{\lambda} \prec \alpha \bm{\mu}$ if and only if $|\lambda_n| = -\lambda_n \leq \alpha$, and $\beta \bm{\mu} \prec \bm{\lambda}$ if and only if 
	\[
	\frac{\lambda_1}{n-1} = \frac{-\lambda_n-\cdots-\lambda_{2}}{n-1} \geq \beta.
	\]
	The assertions then follow directly.
\end{proof}

\begin{lemma}\label{lemma:CompleteBipartite}
	Let $G$ be a simple graph of order $n$ with eigenvalues $\bm{\lambda}=(\lambda_1,\lambda_{2},\dots,\lambda_n)$ and let $\bm{\nu} = (\sqrt{rs},0,\dots,0,-\sqrt{rs})$ for real numbers $r,s>0$. Then,
	\[
	\frac{|\lambda_n|}{\sqrt{rs}}\,\bm{\nu} \,\prec\, \bm{\lambda} \,\prec\, \frac{\|G\|_\ast}{2\sqrt{rs}} \,\bm{\nu}.
	\]
\end{lemma}
\begin{proof}
	Let $\alpha,\beta>0$. We have $\bm{\lambda} \prec \alpha \bm{\nu}$ if and only if 
	\[
	\sum_{j=1}^k \lambda_k \leq \alpha \sqrt{rs} , \qquad (k=1,2,\dots,n-1)
	\]
	and $\beta \bm{\nu} \prec \bm{\lambda}$ if and only if 
	\[
	\sum_{j=1}^k \lambda_k \geq \beta \sqrt{rs}, \qquad (k=1,2,\dots,n-1).
	\]
	Suppose $1\leq \ell\leq n$ is such that $\lambda_1,\dots, \lambda_{\ell} \geq 0$ and $\lambda_{\ell+1},\dots,\lambda_n < 0$. The sum $\sum_{j=1}^k \lambda_k$ is then nondecreasing until $k=\ell$ and nonincreasing afterwards. Hence, $\sum_{j=1}^k \lambda_k$ is maximal at $k=\ell$. However, since $\lambda_1+\cdots + \lambda_n =0$, it follows that
	\begin{align*}
		\|G\|_* &= |\lambda_1| + \cdots + |\lambda_{\ell}| + |\lambda_{\ell+1}| + \cdots + |\lambda_n| \nonumber \\
		&= \lambda_1 + \cdots + \lambda_{\ell} - (\lambda_{\ell+1} + \cdots + \lambda_n) \\
		&= 2(\lambda_1 + \cdots + \lambda_{\ell}),
	\end{align*}
	which implies that $\bm{\lambda} \prec \frac{\|G\|_\ast}{2\sqrt{rs}} \bm{\nu}.$ 
	Moreover, the sum $\sum_{j=1}^k \lambda_k$ is minimal either at $k=1$ or $k=n-1$, where it is respectively equal to $\lambda_1$ and $-\lambda_n$. However, by the Perron--Frobenius theorem, we have $\lambda_1 \geq |\lambda_n|$, which implies that the minimum is realized at $|\lambda_n|$. Therefore, we have $\frac{|\lambda_n|}{\sqrt{rs}}\,\bm{\nu} \prec \bm{\lambda}$, as desired.
\end{proof}

\begin{lemma}\label{lemma:K2}
	Let $G$ be a simple graph of order $n$ with eigenvalues $\bm{\lambda}=(\lambda_1,\lambda_{2},\dots,\lambda_n)$, and let $\bm{w} = (1,\dots,1,-1,\dots,-1)$ if $n$ is even and $\bm{w} = (1,\dots,1,0,-1,\dots,-1)$ if $n$ is odd. Then
	\[
	\beta\bm{w} \,\prec\, \bm{\lambda} \,\prec\, \lambda_1 \bm{w},
	\]
	where
	\[
	\beta:= \begin{cases}
		\frac{2}{n} \sum_{j=1}^{n/2} \lambda_j \quad &\text{if } n\text{ is even;}\\
		\frac{2}{n-1} \sum_{j=1}^{\frac{n-1}{2}} \lambda_j + \frac{2}{n-1}\min\Bigl\{0,\, \lambda_{\frac{n+1}{2}}\Bigr\} \quad &\text{if } n\text{ is odd.}
	\end{cases}
	\]
	Moreover, if $G$ is bipartite, then $\beta = \frac{1}{2\lfloor n/2 \rfloor} \|G\|_\ast$ and
	\[
	\frac{\|G\|_\ast}{2\lfloor n/2 \rfloor}\, \bm{w} \,\prec\, \bm{\lambda} \,\prec\, \lambda_1 \bm{w}.
	\]
\end{lemma}
\begin{proof}
	Write $ s_k:=\sum_{j=1}^k \lambda_j$, where $\lambda_1\ge \cdots \ge \lambda_n$ and $\sum_{j=1}^n \lambda_j=0$. We first prove that $\bm{\lambda}\prec \lambda_1\bm{w}.$ If $n$ is even, then $\bm{\lambda}\prec \alpha \bm{w}$ is equivalent to
	\[
	s_k\le \alpha k \quad (1\le k\le n/2),
	\qquad
	s_k\le \alpha (n-k)\quad (n/2<k<n).
	\]
	Now, $s_k/k$ is the average of the $k$ largest eigenvalues, hence decreases with $k$. Thus, the first family holds if and only if
	\[
	\lambda_1\le \alpha.
	\]
	Also, since $s_k=-\sum_{j=k+1}^n \lambda_j$, the second family is equivalent to
	\[
	\frac{-\lambda_n-\cdots-\lambda_{k+1}}{n-k}\le \alpha,
	\]
	and the left-hand side is the average of the $n-k$ largest entries of $-\bm{\lambda}$. Hence, it is maximized when $n-k=1$. Therefore, the second family holds if and only if
	\[
	-\lambda_n\le \alpha.
	\]
	By Perron--Frobenius, $-\lambda_n\le \lambda_1$. Hence, taking $\alpha=\lambda_1$ gives $\bm{\lambda}\prec \lambda_1\bm{w}.$ 
	If $n$ is odd, the same argument handles the inequalities for
	\[
	1\le k\le \frac{n-1}{2}
	\qquad\text{and}\qquad
	\frac{n+3}{2}\le k<n,
	\]
	and again these reduce to $\lambda_1\le \alpha$. The only extra condition is
	\[
	s_{\frac{n+1}{2}} \le \alpha\, \frac{n-1}{2}.
	\]
	With $\alpha=\lambda_1$, we have
	\[
	s_{\frac{n+1}{2}}
	=
	-\!\sum_{j=(n+3)/2}^n \lambda_j
	\le
	\frac{n-1}{2}\lambda_1,
	\]
	so this middle inequality also holds. Therefore, $\bm{\lambda}\prec \lambda_1\bm{w}$ for all $n$.
	We now prove that $\beta \bm{w}\prec \bm{\lambda}.$ Suppose first that $n$ is even. Then $\beta\bm{w}\prec \bm{\lambda}$ is equivalent to
	\[
	s_k\ge \beta k \quad (1\le k\le n/2),
	\qquad
	s_k\ge \beta(n-k)\quad (n/2<k<n).
	\]
	Since $s_k/k$ decreases with $k$, the first family holds if and only if $\beta\le \frac{2}{n}\sum_{j=1}^{n/2}\lambda_j.$ Using $\sum_{j=1}^n\lambda_j=0$, the second family is equivalent to the same condition. Hence, the largest possible $\beta$ is
	\[
	\beta=\frac{2}{n}\sum_{j=1}^{n/2}\lambda_j.
	\]
	Now suppose that $n$ is odd. Then $\beta\bm{w}\prec \bm{\lambda}$ is equivalent to
	\[
	s_k\ge \beta k \quad \bigl(1\le k\le \tfrac{n-1}{2}\bigr),\qquad s_k\ge \beta(n-k)\quad \bigl(\tfrac{n+3}{2}\le k<n\bigr)
	\]
	and
	\[
	s_{ \frac{n+1}{2}}\ge \beta \frac{n-1}{2}.
	\]
	As above, the first two families are both equivalent to
	\[
	\beta\le \frac{2}{n-1}\sum_{j=1}^{(n-1)/2}\lambda_j.
	\]
	Moreover, the middle inequality is equivalent to
	\[
	\beta\le \frac{2}{n-1}\sum_{j=1}^{(n+1)/2}\lambda_j
	=
	\frac{2}{n-1}\sum_{j=1}^{(n-1)/2}\lambda_j
	+\frac{2}{n-1}\lambda_{(n+1)/2}.
	\]
	Therefore, the optimal value for $\beta$ is
	\[
	\beta
	=
	\frac{2}{n-1}\sum_{j=1}^{(n-1)/2}\lambda_j
	+
	\frac{2}{n-1}\min\{0,\lambda_{(n+1)/2}\}.
	\]
	Finally, assume that $G$ is bipartite. Then $\lambda_j=-\lambda_{n+1-j}$ for all $1\le j\le n$. If $n$ is even, this gives
	\[
	\frac{\|G\|_*}{n}
	=
	\frac{1}{n}\sum_{j=1}^n |\lambda_j|
	=
	\frac{2}{n}\sum_{j=1}^{n/2}\lambda_j
	=
	\beta.
	\]
	If $n$ is odd, then necessarily $\lambda_{(n+1)/2}=0$, and similarly
	\[
	\frac{\|G\|_*}{n-1}
	=
	\frac{2}{n-1}\sum_{j=1}^{(n-1)/2}\lambda_j = \beta.
	\]
	Hence, in both cases
	\begin{equation*}
		\frac{\|G\|_*}{2\lfloor n/2 \rfloor}\,\bm{w}\prec \bm{\lambda}\prec \lambda_1\bm{w},
	\end{equation*}
	as desired.
\end{proof}

Using the preceding lemmas, we now prove our main theorem, which provides a general mechanism for generating sharp inequalities on simple graphs from positive Schur-convex functions and the majorization relations established above.

\begin{theorem}\label{thm - main}
	Let $G$ be a simple graph of order $n$ with eigenvalues $\bm{\lambda}=(\lambda_1,\lambda_2,\dots,\lambda_{n})$, and let $\varphi$ be a positive homogeneous function on graphs which is Schur-convex relative to its eigenvalues. Then, for all integers $r,s>0$ such that $r+s=n$, 
	\begin{equation}\label{eq - main}
		\frac{\beta}{n-1} \leq \frac{\varphi(G)}{(n-1)\varphi(M_{n})} \leq \frac{\lambda_1}{n-1} \leq \frac{\varphi(G)}{\varphi(K_{n})} \leq |\lambda_{n}| \leq \frac{\sqrt{rs}\varphi(G)}{\varphi(K_{r,s})} \leq \frac{\|G\|_*}{2},
	\end{equation}
	where
	\[
	\beta:= \begin{cases}
		\frac{2}{n} \sum_{j=1}^{n/2} \lambda_j \quad &\text{if } n\text{ is even;}\\
		\frac{2}{n-1} \sum_{j=1}^{(n-1)/2} \lambda_j + \frac{2}{n-1}\min\Bigl\{0,\, \lambda_{\frac{n+1}{2}}\Bigr\} \quad &\text{if } n\text{ is odd.}
	\end{cases}
	\]
	In particular, if $G$ is bipartite, then the inequality chain reduces to 
	\[
	\frac{\|G\|_*}{2\lfloor n/2\rfloor} \leq \frac{\varphi(G)}{\varphi(M_{n})} \leq \lambda_1 \leq \frac{\sqrt{rs}\varphi(G)}{\varphi(K_{r,s})} \leq \frac{\|G\|_*}{2}.
	\]
	Moreover, in \eqref{eq - main}, the first and second inequalities are saturated by $M_{n}$, the third and fourth are saturated by $K_{n}$ and the last two are saturated by $K_{r,s}$.
\end{theorem}
\begin{proof}
	The result follows by using Schur-convexity with \Cref{lemma:Complete,lemma:CompleteBipartite,lemma:K2}, and then rearranging.
	Now, noting that $|\lambda_n|=\lambda_1$ and $\beta=\frac{\|G\|_*}{2\lfloor n/2 \rfloor}$ for bipartite graphs, it follows that
	\[
	\frac{\|G\|_*}{2(n-1)\lfloor n/2\rfloor} \leq \frac{\varphi(G)}{(n-1)\varphi(M_{n})} \leq \frac{\lambda_1}{n-1} \leq \frac{\varphi(G)}{\varphi(K_{n})} \leq \lambda_{1} \leq \frac{\sqrt{rs}\varphi(G)}{\varphi(K_{r,s})} \leq \frac{\|G\|_*}{2}.
	\]
	However, it is straightforward to show that $\bm{\lambda}(M_n) \prec \bm{\lambda}(K_n)$ and $\bm{\lambda}(K_n) \prec (n-1)\bm{\lambda}(K_{1,1})$, which implies that $\varphi(M_n) \leq \varphi(K_n)$ and $\varphi(K_n) \leq (n-1)\varphi(K_{1,1})$ for all positive Schur-convex function $\varphi$. Consequently, the second inequality is a sharper lower bound for $\lambda_1$ than the fourth, and the fifth inequality is a sharper upper bound than the third. Therefore, we may only consider these inequalities without any loss of generality.
\end{proof}

Naturally, one may seek Schur-convex functions for which each inequality is as tight as possible. The rest of the paper is devoted to this objective. We begin, however, with an example that proves important.

\subsection{Schatten $p$-norms}\label{sec - Schatten}

We first illustrate \Cref{thm - main} with the Schatten $p$-norms. Recall that, for a graph $G$ with eigenvalues $\lambda_1 \geq \lambda_2 \geq \cdots \geq \lambda_n$, the Schatten $p$-norm is given by
\[
\|G\|_p
=
\left(|\lambda_1|^p+|\lambda_2|^p+\cdots+|\lambda_n|^p\right)^{1/p}.
\]
Since this function is convex and symmetric for every $p\geq 1$, it is Schur-convex. Hence, applying \Cref{thm - main} to the Schatten $p$-norm gives, for every simple graph $G$ of order $n$,
\begin{equation}\label{ineq - Schatten}
	\frac{\beta^p}{(n-1)^p} \leq \frac{\norm{G}_p^p}{2(n-1)^p\lfloor \frac{n}{2}\rfloor} \leq \frac{\lambda_1^p}{(n-1)^p} \leq \frac{\norm{G}_p^p}{ (n-1)^p+n-1} \leq |\lambda_{n}|^p \leq \frac{\|G\|_p^p}{2} \leq \frac{\norm{G}_*^p}{2^p}.
\end{equation}
Here, we used the identities
\[
\|M_n\|_p^p=2\lfloor n/2\rfloor,
\qquad
\|K_{r,s}\|_p^p=2(rs)^{p/2}
\quad\text{ and }\quad
\|K_n\|_p^p=(n-1)^p+n-1.
\]

Most inequalities in \eqref{ineq - Schatten} are classical or follow from standard results such as Hölder's inequality and the monotonicity of power means. The proof given here, however, relies on a general framework that may be viewed as more fundamental, since a single argument yields all the bounds. Moreover, the Schatten $p$-norms are particularly convenient because they encode important geometric properties of the graph. Indeed, for every even integer $d\geq 2$,
\[
\|G\|_d^d
=
\lambda_1^d+\lambda_2^d+\cdots+\lambda_n^d
=
c_d(G),
\]
where $c_d(G)$ denotes the number of closed walks of length $d$ in $G$. Therefore, \eqref{ineq - Schatten} immediately yields the following formulation in terms of closed walks.

\begin{corollary}\label{cor - Schatten}
	Let $G$ be a simple graph of order $n$ with eigenvalues $\bm{\lambda}=(\lambda_1,\dots,\lambda_{n})$. Then, for all even $d\geq 2$,
	\begin{equation}\label{eq - Schatten}
		\frac{\beta^d}{(n-1)^d} \leq \frac{c_d(G)}{2(n-1)^d\lfloor \frac{n}{2}\rfloor} \leq \frac{\lambda_1^d}{(n-1)^d} \leq \frac{c_d(G)}{(n-1)^d+n-1} \leq |\lambda_{n}|^d \leq \frac{c_d(G)}{2} \leq \frac{\norm{G}_*^d}{2^d},
	\end{equation}
	where
	\[
	\beta := \begin{cases}
		\frac{2}{n} \sum_{j=1}^{n/2} \lambda_j \quad &\text{if } n\text{ is even;}\\
		\frac{2}{n-1} \sum_{j=1}^{(n-1)/2} \lambda_j + \frac{2}{n-1}\min\Bigl\{0,\, \lambda_{\frac{n+1}{2}}\Bigr\} \quad &\text{if } n\text{ is odd.}
	\end{cases}
	\]
	In particular, if $G$ is bipartite, then the inequality chain reduces to 
	\[
	\frac{\|G\|_*^d}{2^d\lfloor n/2 \rfloor^d} \leq \frac{c_d(G)}{2\lfloor n/2\rfloor} \leq \lambda_1^d \leq \frac{c_d(G)}{2} \leq \frac{\norm{G}_*^d}{2^d}.
	\]
\end{corollary}

\begin{remark}
	Although the case $d=2$ recovers known estimates (see Caporossi, Cvetkovi\'c, Gutman and Hansen~\cite{Caporossi}, and Bhattacharya, Friedland and Peled~\cite{Bhattacharya}), the higher even-degree formulation appears not to have been recorded in this generality. In particular, the sharp lower bound on $|\lambda_n|^d$ in terms of $c_d(G)$, with equality at $K_n$, seems to be new.
\end{remark}

\subsection{Random vector norms}

Recall that our main objective is to make the inequalities in \Cref{thm - main} as sharp as possible. This problem is difficult since the existing characterizations of Schur-convex functions appear to be unfit for this analysis. However, as stated in \Cref{sec - random vector norms}, the random vector norms constitute a large class of positive functions which are Schur-convex in the eigenvalues of a graph. They can thus serve as a simplified family on which we may do our analysis. 

In fact, the Schatten $p$-norms considered above are a special limit case of these norms when $p=d$ is even. Indeed, consider the random variables $B_n\sim\mathrm{Bernoulli}(p_n)$, where 
\[
p_n:=\frac{1}{n^d},
\]
and define $X_n:=nB_n$. It is a standard fact that $\kappa_m(aY)=a^m\kappa_m(Y)$. Hence,
\[
\kappa_m(X_n)=\kappa_m(nB_n) = n^m \kappa_m(B_n).
\]
Moreover, for a Bernoulli$(p)$ random variable $B$, we have
\[
\kappa_m(B) = \sum_{j=1}^m (-1)^{j+1} (j-1)! S(m,j)p^j,
\]
where the $S(m,j)$ are the Stirling numbers of the second kind. Hence, as $p\to 0$, one has $\kappa_m(B)=p+O(p^2)$ for all $m\geq 1$, which implies $\kappa_m(B_n)=n^{-d}+O\bigl(n^{-2d}\bigr).$ Therefore,
\[
\kappa_m(X_n)
=n^m\!\left(n^{-d}+O\bigl(n^{-2d}\bigr)\!\right)
= n^{m-d}+O\!\left(n^{m-2d}\right).
\]
In particular,
\[
\kappa_m(X_n)\xrightarrow{n\to \infty} 0 \quad (1\le m\le d-1),
\qquad
\kappa_d(X_n) \xrightarrow{n\to \infty} 1\neq 0.
\]

Now, let $\boldsymbol{\pi}=(\pi_1,\dots,\pi_r)$ be a partition of $d$ with $\boldsymbol{\pi}\neq (d)$. Since $\boldsymbol{\pi}\neq (d)$, at least one part $\pi_i$ is strictly smaller than $d$. Hence, at least one factor in the product $\prod_{i=1}^r \kappa_{\pi_i}(X_n)$ converges to $0$.
Consequently,
\[
\kappa_{\pi_1}(X_n)\cdots \kappa_{\pi_r}(X_n) \xrightarrow{n\to \infty} 0
\]
for every partition $\boldsymbol{\pi}\neq (d)$. Therefore, considering such a sequence of random variables, we have
\begin{align*}
	\|G\|_{X_m,d}^d &= \sum_{\boldsymbol{\pi}\vdash d} \frac{\kappa_{\boldsymbol{\pi}}(X_m) }{ y_{\boldsymbol{\pi}}} \,c_{\boldsymbol{\pi}} \xrightarrow{m\to \infty} \frac{c_d}{y_{(d)}} = \frac{\|G\|_d^d}{d!}.
\end{align*}

As with the Schatten norms, which exploits \Cref{thm - main} to connect the spectral data of a graph to the numbers of closed walks $c_d$, the identity
\[
\|G\|_{X,d}^d = \sum_{\boldsymbol{\pi}\vdash d} \frac{\kappa_{\boldsymbol{\pi}} c_{\boldsymbol{\pi}}}{ y_{\boldsymbol{\pi}}}
\]
in \eqref{thm:Cumulants} yields further bounds of this type. Although the result is immediate, we state it separately for clarity, at the cost of some mild repetition.

\begin{proposition}\label{prop - random}
	Let $G$ be a simple graph of order $n$ with eigenvalues $\bm{\lambda}=(\lambda_1,\lambda_2,\dots,\lambda_{n})$, and let $X$ be a nonconstant random variable with $\mathbb{E}|X|^d<\infty$. Then, for all even $d\geq 2$,
	\begin{equation}\label{eq - random}
		\frac{\beta}{n-1} \leq \frac{\|G\|_{X,d}}{(n-1)\|M_n\|_{X,d}} \leq \frac{\lambda_1}{n-1} \leq \frac{\|G\|_{X,d}}{\|K_n\|_{X,d}} \leq |\lambda_{n}| \leq \frac{\sqrt{rs}\|G\|_{X,d}}{\|K_{r,s}\|_{X,d}} \leq \frac{\|G\|_*}{2},
	\end{equation}
	where
	\[
	\beta:= \begin{cases}
		\frac{2}{n} \sum_{j=1}^{n/2} \lambda_j \quad &\text{if } n\text{ is even;}\\
		\frac{2}{n-1} \sum_{j=1}^{(n-1)/2} \lambda_j + \frac{2}{n-1}\min\Bigl\{0,\, \lambda_{\frac{n+1}{2}}\Bigr\} \quad &\text{if } n\text{ is odd.}
	\end{cases}
	\]
	In particular, if $G$ is bipartite, then the inequality chain reduces to 
	\[
	\frac{\|G\|_*}{2\lfloor \frac{n}{2}\rfloor} \leq \frac{\|G\|_{X,d}}{\|M_n\|_{X,d}} \leq \lambda_{1} \leq \frac{\sqrt{rs}\|G\|_{X,d}}{\|K_{r,s}\|_{X,d}} \leq \frac{\|G\|_*}{2}.
	\]
	Moreover, in \eqref{eq - random}, the first and second inequalities are saturated by $M_{n}$, the third and fourth are saturated by $K_{n}$ and the last two are saturated by $K_{r,s}$.
\end{proposition}

\section{Small values of $d$}\label{sec - d small}

Recall that we seek to make the inequalities in \Cref{thm - main} as sharp as possible. More precisely, we shall focus on this question in the case of random vector norms, which depends on a parameter $d$. Since
\[
\|G\|_{X,d}^d = \sum_{\boldsymbol{\pi}\vdash d} \frac{\kappa_{\boldsymbol{\pi}} c_{\boldsymbol{\pi}}}{ y_{\boldsymbol{\pi}}},
\]
the information required to compute a random vector norm for a given $d\geq 2$ is essentially the number of closed walks $c_k$ for even $k\leq d$. Therefore, we may naturally reframe our problem as follows: Given an even $m\geq 2$, which function optimizes each inequality in \Cref{prop - random}, in the closure of the union of the sets of random vector norms for $d\in\{2,4,\dots,m\}$? Providing an answer to this question could provide new sharp inequalities relating the numbers of closed walks up to a given order and the spectral information of a graph.

We do not have a complete answer to this question. However, for small values of $d$, we may provide tight inequalities. Indeed, the case $d=2$ is almost trivial since in that case, \eqref{thm:Cumulants} ensures that
\begin{equation*}
	\norm{G}_{X,2}^2=\sum_{\boldsymbol{\pi}\vdash 2} \frac{\kappa_{\boldsymbol{\pi}} c_{\boldsymbol{\pi}} }{ y_{\boldsymbol{\pi}}} = \frac{\kappa_1^2 c_1^2}{(1!)^2\, 2!} + \frac{\kappa_2 c_2}{(2!)^1 \, 1!}.
\end{equation*} 
Recall that for simple graphs, $c_1=0$. Hence, we have
\[
\norm{G}_{X,2}^2= \frac{\kappa_2 c_2}{2} = \frac{\kappa_2}{2} \|G\|_2^2.
\]
Therefore, the case of $d=2$ always corresponds to a multiple of the Schatten 2-norm and there is only one possible function to consider.

\subsection{The case $d=4$}

In this case, \cref{thm:Cumulants}, together with the fact that $c_1$ is always equal to $0$, yields
\[
\norm{G}_{X,4}^4 = \frac{\kappa_2^2 c_2^2}{(2!)^2\, 2!} + \frac{\kappa_4 c_4}{(4!)^1 \, 1!} = \frac{\kappa_2^2 c_2^2}{8} + \frac{\kappa_4 c_4}{24}. 
\]Therefore, to make the inequalities in \Cref{thm - main} as sharp as possible, we seek to optimize the ratio
\[
\frac{\norm{G}_{X,4}^4}{\norm{G_2}_{X,4}^4} = \frac{3\kappa_2^2 c_2^2 + \kappa_4 c_4}{3\kappa_2^2 c_2^2(G_2) + \kappa_4 c_4(G_2)},
\]
where $G_2$ is either $M_n$, $K_n$ or $K_{1,1}$. Recall that $\kappa_2 = \sigma^2$, the variance of the random variable $X$. As such, we have $\kappa_2>0$ since $X$ is nonconstant. We write
\[
\frac{\norm{G}_{X,4}^4}{\norm{G_2}_{X,4}^4} = \frac{3 c_2^2 + \frac{\kappa_4}{\kappa_2^2} c_4}{3 c_2^2(G_2) + \frac{\kappa_4}{\kappa_2^2} c_4(G_2)}.
\]
Thus, we want to optimize the right-hand side, which corresponds to the function 
\[
f(x) = \frac{3 c_2^2 + c_4 x}{3 c_2^2(G_2) + c_4(G_2) x}
\]
Moreover, $\kappa_4 = \mu_4 - 3\sigma^4$, where $\mu_k = \mathbb{E}\bigl[(X-\mathbb{E}[X])^k\bigr]$ is the $k$th central moment of $X$. Cauchy--Schwarz ensures that $\mu_4 \geq \mu_2^2 = \sigma^4.$ Therefore, it follows that
\[
\kappa_4 = \mu_4 - 3\sigma^4 \geq -2\sigma^4 = -2\kappa_2^2,
\]
which means that $x = \kappa_4/\kappa_2^2 \geq -2$. Hence, we want to optimize $f$ over the domain $[-2,\infty)$. Since $f$ is a quotient of two linear functions, it is easily verified that $f$ is monotone. Therefore, the function attains its minimum and maximum at the endpoints of its domain, namely 
\[
f(-2) = \frac{3 c_2^2 -2 c_4}{3 c_2^2(G_2) - 2c_4(G_2)} \qquad \text{and}\qquad \lim_{x\to\infty} f(x) = \frac{c_4}{c_4(G_2)}.
\]
These two extremal cases correspond, respectively, to the random vector norm induced by the Rademacher distribution and to the Schatten $4$-norm. Hence, when $d=4$, the sharpest inequalities obtained from \Cref{prop - random} are always realized by one of these two choices. 

To state the result, we write $s^+ := \max\{0,s\}$ for the positive part of a real number, and we recall that the \emph{first Zagreb index} of a graph $G$ is $Z_1 := \sum_i d_i^2$, where $d_i$ denotes the degree of vertex $i$. This invariant is closely related to the number of closed walks of length $4$: if $m$ is the number of edges of $G$ and $q$ the number of squares (i.e., the number of $4$-cycles), then
\[
c_4 = 2Z_1 - 2m + 8q.
\]
This identity is useful in proving the following theorem.

\begin{theorem}\label{thm - d=4}
	Let $G$ be a simple graph of order $n$ with eigenvalues $\bm{\lambda}=(\lambda_1,\dots,\lambda_{n})$. Then\vspace{3pt}
	
	\begin{minipage}{.46\textwidth}
		\begin{enumerate}
			\setlength{\itemindent}{-12pt}
			\item[(a)] $\frac{c_{4}}{2\lfloor n/2 \rfloor} \leq \lambda_1^4 \leq \frac{c_4}{1+(n-1)^{-3}}$,
			\item[(b)] $ 8c_4+6(c_{2}^{2}-2c_{4})^+ \leq \|G\|_*^4$,\vspace{5pt}
		\end{enumerate}
	\end{minipage}%
	\begin{minipage}{.54\textwidth}
		\begin{enumerate}
			\setlength{\itemindent}{-12pt}
			\item[(c)] $\frac{3c_{2}^{2}-2c_{4}}{n (n - 1) (n^2 + 3 n - 6)} \leq \lambda_{n}^4 \leq \frac{4c_4-3( 2c_4-c_2^2)^{\!+}}{8}$,
			\item[(d)] $\beta^4 \leq \frac{3c_{2}^{2}-2c_{4}}{12\lfloor n/2 \rfloor^2-4\lfloor n/2 \rfloor}$,%\vspace{5pt}
		\end{enumerate}
	\end{minipage}
	
	\noindent where
	\[
	\beta:= \begin{cases}
		\frac{2}{n} \sum_{j=1}^{n/2} \lambda_j \quad &\text{if } n\text{ is even;}\\
		\frac{2}{n-1} \sum_{j=1}^{(n-1)/2} \lambda_j + \frac{2}{n-1}\min\Bigl\{0,\, \lambda_{\frac{n+1}{2}}\Bigr\} \quad &\text{if } n\text{ is odd.}
	\end{cases}
	\]
	In particular, if $G$ is bipartite, then the inequalities reduce to 
	
	\begin{minipage}{.46\linewidth}
		\begin{enumerate}
			\setlength{\itemindent}{-12pt}
			\item[(a)] $ \frac{c_{4}}{2\lfloor n/2 \rfloor} \leq \lambda_{1}^4 \leq \frac{4c_4-3( 2c_4-c_2^2)^{\!+}}{8}$,\vspace{7pt}
		\end{enumerate}
	\end{minipage}%
	\begin{minipage}{.51\linewidth}
		\begin{enumerate}
			\setlength{\itemindent}{-12pt}
			\item[(b)] $8c_4+6(c_{2}^{2}-2c_{4})^+ \leq \|G\|_*^4 \leq \frac{4\lfloor n/2 \rfloor^3 (3c_{2}^{2}-2c_{4})}{3\lfloor n/2 \rfloor-1}$.\vspace{-3pt}
		\end{enumerate}
	\end{minipage}
	
	\noindent Moreover, all of the inequalities are sharp and they are best possible among the random vector norms with $d\in\{2,4\}$.
\end{theorem}
\begin{proof}
	By the above discussion, the minimum and maximum of $\norm{G}_{X,4}^4/\norm{G_2}_{X,4}^4 $ is realized by
	\[
	\frac{3 c_2^2 -2 c_4}{3 c_2^2(G_2) - 2c_4(G_2)} \qquad \text{or}\qquad \frac{c_4}{c_4(G_2)}.
	\]
	From the ordering of the $\ell^p$ vector norms, we have $c_2^2(G_2) \geq c_4(G_2)$ for any graph $G_2$. Hence,
	\[
	3 c_2^2(G_2) - 2c_4(G_2) \geq 3 c_2^2(G_2) -2 c_2^2(G_2) = c_2^2(G_2) >0,
	\]
	and it follows that 
	\begin{equation}\label{eq - order d=4}
		\frac{3c_{2}^{2}-2c_{4}}{3c_{2}^{2}(G_{2})-2c_{4}(G_{2})} \geq \frac{c_{4}}{c_{4}(G_{2})}
	\end{equation}
	if and only if 
	\begin{equation}\label{eq - order2}
		c_{2}^{2}c_{4}(G_{2}) \geq c_{4}c_{2}^{2}(G_{2}).
	\end{equation}
	\medskip
	First consider the case $G_2=M_n$, where $c_2(G_2)=c_4(G_2)=2\lfloor n/2 \rfloor$ and \eqref{eq - order2} becomes
	\[
	c_{2}^{2} \geq 2\lfloor n/2 \rfloor c_{4}.
	\]
	We claim this is never satisfied, except in the equality case, where the maximum equals the minimum. Indeed, let $r:=\lfloor n/2\rfloor$, $m$ be the number of edges in the graph and $q$ the number of squares in the graph. Then we have $c_2=2m$ and $c_4=2Z_1-2m+8q$, and it is enough to show that
	\begin{equation}\label{eq - zag}
		m^2\le r(Z_1-m),
	\end{equation}
	because then
	\[
	c_2^2=4m^2\le 4r(Z_1-m)\le 2rc_4.
	\]
	First assume that $m\le r$. Since $Z_1\ge \sum_i d_i=2m$, it then follows that
	\[
	r(Z_1-m)\ge r(2m-m) = rm\ge m^2,
	\]
	which proves \eqref{eq - zag}.
	Assume now that $m=r+1$. If $n$ is even, then $n=2r$. By Cauchy--Schwarz,
	\[
	Z_1 = \sum_{i} d_i^2 \ge \frac{1}{n} \Big( \sum_{i} d_i \Big)^{\!2} = \frac{(2m)^2}{n}=\frac{2m^2}{r}.
	\]
	Hence, since $m>r$,
	\[
	r(Z_1-m)\ge 2m^2-rm>m^2.
	\]
	Suppose on the other hand that $n$ is odd, that is $n=2r+1$. Since $\sum_{i=1}^n d_i = 2m = 2r+2$, the degree sequence $\bm d =(d_1,\dots,d_n)$ majorizes $\bm b =(2,1,\dots,1)$, which is the most balanced nonnegative integer vector of length $2r+1$ with sum $2r+2$. Since $x\mapsto x^2$ is convex, the function $\bm x\mapsto \sum_i x_i^2$ is Schur-convex. Hence,
	\[
	Z_1 = \sum_{i=1}^{n} d_i^2 \ge \sum_{i=1}^{n} b_i^2 = 2^2+(2r)\cdot 1^2 = 2r+4,
	\]
	which implies that
	\[
	r(Z_1-m)\ge r\bigl((2r+4)-(r+1)\bigr)=r(r+3)\ge (r+1)^2=m^2.
	\]
	Now, if $m\ge r+2$, then Cauchy--Schwarz gives $Z_1\ge \frac{4m^2}{2r+1}$, and therefore,
	\[
	r(Z_1-m)\ge \frac{4rm^2}{2r+1}-rm.
	\]
	Observe that
	\[
	\frac{4rm^2}{2r+1}-rm\ge m^2 \quad\iff\quad m\ge \frac{r(2r+1)}{2r-1}.
	\]
	Since $r\ge 1$ and $m\ge r+2$, this always holds, and it thus follows that we have
	\[
	\frac{3c_{2}^{2}-2c_{4}}{3c_{2}^{2}(M_n)-2c_{4}(M_n)} \leq \frac{c_{4}}{c_{4}(M_n)}.
	\]
	
	\smallskip
	Now, for the case of $G_2=K_n$, the condition \eqref{eq - order2} becomes after simplifying
	\begin{equation}\label{eq - K_n}
		(n^2-3n+3)c_{2}^{2} \geq n(n-1)c_{4}.
	\end{equation}
	However, Dalèn proved in \cite{Dalen} that for all $\bm{x}\in\mathbb{R}^n$, we have
	\[
	\sum_{i=1}^n (x_i-\overline{\bm{x}})^4 \leq \frac{n^2-3n+3}{n(n-1)} \!\left(\sum_{i=1}^n (x_i-\overline{\bm{x}})^2 \right)^{\!\!2},
	\]
	where $\overline{\bm{x}} = \frac{x_1+x_2+\cdots+x_n}{n}$. Applying this bound with $x_i=\lambda_i$ yields \eqref{eq - K_n} directly. Therefore, the ordering \eqref{eq - order d=4} holds when $G_2=K_n$. 
	
	\smallskip 
	Lastly, if $G_2=K_{1,1}$, then $c_2(G_2)=c_4(G_2)=2$ and the condition becomes
	\[
	c_{2}^{2} > 2 c_{4}.
	\]
	In this case, the condition may or may not hold, so we need to consider both cases. The result then follows directly by taking either the random vector norm induced by the Rademacher distribution or the Schatten 4-norm in the inequalities of \Cref{thm - main}.
\end{proof}

\Cref{thm - d=4} can be interpreted as a fourth-order refinement of several classical spectral estimates. Indeed, the case \(d=2\) recovers inequalities depending only on \(c_2=2m\), such as the usual lower bounds for the graph energy and the standard estimates for the extremal eigenvalues. By contrast, the bounds above refine edge-count estimates by incorporating closed walks of length \(4\). Hence, \Cref{thm - d=4} does two things: it gives sharp fourth-order analogues of classical bounds, and it shows that these inequalities arise from a single majorization and Schur-convexity framework rather than from separate extremal arguments.

\subsection{Bound for spectral radius in terms of squares} 
As a corollary of the last result, we establish a bound for the smallest and largest eigenvalues of any graph in terms of squares. To this end, suppose $G$ is a graph of order $n$ with $m$ edges and $q$ squares, which implies that $c_4=2Z_1-2m+8q$, in which $Z_1=\sum_i d_i^2$ is the first Zagreb index. Zhou established in \cite{Zhou} that $Z_1$ satisfies
\begin{align}
	Z_1\leq n(2m-n+1)\label{eq:ZhouExp}
\end{align} 
whenever $G$ is connected, with equality when $G=K_n$. Using this bound, we have the following result only using the number of vertices, edges, and squares in $G$.

\begin{theorem}
	Suppose $G$ is a connected graph of order $n$ with $m$ edges and $q$ squares. Then the extremal eigenvalues of $G$ satisfy the following bounds, which are attained by $K_n$:
	\begin{align*}
		\sqrt[4]{\frac{4\bigl(m(3m+1)-2mn+n(n-1)-4q\bigr)}{n (n - 1) (n^2 + 3 n - 6)}} \leq |\lambda_n| \leq \lambda_1\leq \sqrt[4]{\frac{8q-2m+2n(2m-n+1)}{1+(n-1)^{-3}}}.
	\end{align*}
\end{theorem}

\begin{proof}
	Applying the relation $c_4 = 2Z_1-2m+8q$ and Zhou's bound in \eqref{eq:ZhouExp} to the upper bound of $\lambda_1$ in \Cref{thm - d=4}, we get 
	\begin{align*}
		\lambda_1&\leq \sqrt[4]{\frac{c_4}{1+(n-1)^{-3}}} = \sqrt[4]{\frac{2Z_1-2m+8q}{1+(n-1)^{-3}}} \leq \sqrt[4]{\frac{8q-2m+2n(2m-n+1)}{1+(n-1)^{-3}}}.
	\end{align*}
	Likewise, doing the same to the lower bound of $|\lambda_n|$ in \Cref{thm - d=4}, we obtain
	\begin{align*}
		|\lambda_n| &\geq \sqrt[4]{\frac{3c_{2}^{2}-2c_{4}}{n (n - 1) (n^2 + 3 n - 6)}} = \sqrt[4]{\frac{12m^2-2(2Z_1-2m+8q)}{n (n - 1) (n^2 + 3 n - 6)}} \\
		&\geq \sqrt[4]{\frac{12m^2-2(8q-2m+2n(2m-n+1))}{n (n - 1) (n^2 + 3 n - 6)}} \\
		&= \sqrt[4]{4\,\frac{m(3m+1)-2mn+n(n-1)-4q}{n (n - 1) (n^2 + 3 n - 6)}}.
	\end{align*}
	Finally, each inequality is attained by $K_n$.
\end{proof}

%%%%%%%%%%%%%%%%%%%%%%%%%%%%%%%%%%%%%%%%%%%%%%%%%%%%%%%%%%%%%%%%%%%%%%%%%%%
%%%%%%%%%%%%%%%%%%%%%%%%%%%%%%%%%%%%%%%%%%%%%%%%%%%%%%%%%%%%%%%%%%%%%%%%%%%
%%%%%%%%%%%%%%%%%%%%%%%%%%%%%%%%%%%%%%%%%%%%%%%%%%%%%%%%%%%%%%%%%%%%%%%%%%%
\subsection{Spectral radii of square-free bipartite graphs}

Let us now establish a bound for the spectral radius of any square-free bipartite graph. Suppose $G$ is a graph of order $n$ with $m$ edges. A bound for $Z_1$ due to Zhou and Stevanovi\'c \cite{ZhouBipartite} states that
\begin{align}
	Z_1\leq n(n-1) \label{eq:ZagrebBound1}
\end{align} whenever $G$ contains no triangle or square and $n\geq 2$. Moreover, equality is attained if and only if $G=S_n=K_{1,n-1}$ or a Moore graph of diameter 2.

\begin{theorem}\label{thm:BipartiteRadius}
	Suppose $B$ is a square-free bipartite graph of order $n$ with $m$ edges. The spectral radius of $B$ satisfies the following bound, which is attained by $S_n$:
	\begin{align*}
		\lambda_1\leq \sqrt[4]{ n(n-1)-m}.
	\end{align*} 
\end{theorem}

\begin{proof}
	The relation $c_4=2Z_1-2m+8q= 2Z_1-2m$ and \Cref{thm - d=4} imply that
	\begin{align*}
		\lambda_{1} \leq \sqrt[4]{Z_1-m-\tfrac{3}{2}( Z_1-m-m^2)^{+}} \leq \sqrt[4]{Z_1-m}.
	\end{align*}
	$B$ contains no triangle since it is bipartite. We can thus apply the bound in \eqref{eq:ZagrebBound1} to conclude
	\begin{align*}
		\lambda_1\leq \sqrt[4]{n(n-1)-m}.
	\end{align*} 
	Note that the bounds of \Cref{thm - d=4} and \eqref{eq:ZagrebBound1} both hold when $B=S_n$.
\end{proof}

We remark that $\sqrt[4]{n(n-1)-m}\leq \sqrt{m}$ if and only if $m\geq n-1$. Therefore, the inequality in \Cref{thm:BipartiteRadius} is sharper than the classic bound $\lambda_1\leq \sqrt{m}$ provided $B$ has enough edges. In particular, the new bound is an improvement whenever $B$ is a \emph{connected} bipartite graph.

\subsection{The case $d=6$}

Naturally, one may seek to explore the optimality of the bounds in \Cref{thm - main} in the case $d=6$. A direct computation reveals that 
\[
\|G\|_{X,6}^6 = \frac{\frac{1}{48}\kappa_{2}^3 c_2^3 + \frac{1}{48}\kappa_2 \kappa_4 c_2c_4 + \frac{1}{72}\kappa_3^2 c_3^2 +\frac{1}{720} \kappa_6 c_6}{\frac{1}{48}\kappa_{2}^3 c_2^3(G_2) + \frac{1}{48}\kappa_2 \kappa_4 c_2(G_2)c_4(G_2) + \frac{1}{72}\kappa_3^2 c_3^2(G_2) +\frac{1}{720} \kappa_6 c_6(G_2)}.
\]
While it is still possible to optimize each of the quantities $\frac{\|G\|_{X,d}}{\|G_2\|_{X,d}}$ for $G_2=M_n$, $K_n$ and $K_{r,s}$, the growing number of variables (i.e., $\kappa_2$, $\kappa_3$, $\kappa_4$ and $\kappa_6$), along with the conditions for them to be cumulants of a valid distribution, makes the problem substantially more complex. For instance, a detailed and tedious analysis reveals that the maximum of $\frac{\sqrt{rs}\|G\|_{X,d}}{\|K_{r,s}\|_{X,d}}$ is realized by different functions in three distinct regions. Namely, the maximum is
\begin{equation}\label{eq - cases d=6}
	\begin{cases}
		\frac{c_{6}}{2} ~~~ &\text{if } c_{2}^{2}-2c_{4}<0\le6c_{6}-3c_{2}c_{4}-2c_{3}^{2},\\
		\frac{c_{6}}{2} + \frac{15c_2(c_{2}^{2}-2c_{4})}{32} ~~~ &\text{if } 2(3c_{2}c_{4}+2c_{3}^{2}-6c_{6})^+\le3c_{2}(c_{2}^{2}-2c_{4}),\\
		\frac{c_{6}}{2}+\frac{5(2c_{3}^{2}+3c_{2}c_{4}-6c_{6})^{2}}{8(8c_{3}^{2}+18c_{2}c_{4}-3c_{2}^{3}-24c_{6})} ~~~ &\text{if } 3c_{2}(c_{2}^{2}-2c_{4})^+ < 2(3c_{2}c_{4}+2c_{3}^{2}-6c_{6}),
	\end{cases}
\end{equation}
which leads to the sharp inequality
\[
\|G\|_{\ast}^6 \geq 32c_{6} + 10\max\!\left\{0,\, 3c_2(c_{2}^{2}-2c_{4}),\, \frac{4(3c_{2}c_{4}+2c_{3}^{2}-6c_{6})^{2}}{8c_{3}^{2}+18c_{2}c_{4}-3c_{2}^{3}-24c_{6}} \right\}.
\]

Moreover, a similar result holds for the minimum. However, the minimum is attained by four different functions in four regions instead of three, distinct from the ones above. 

To completely solve this optimization problem, a similar analysis would need to be done for $G_2=M_n$ and $K_n$ as well, which quickly becomes difficult and, more importantly, impractical to work with in concrete settings where one would want to use these inequalities. Therefore, we shall focus on the first two regions described above, which corresponds to the Schatten $p$-norm and the random vector norm induced by the Rademacher distribution, as in the case $d=4$ of \Cref{thm - d=4}.

Moreover, let us make a key observation. Write
\[
A \;=\; 3c_{2}c_{4}+2c_{3}^{2}-6c_{6}, \qquad B \;=\; c_2^2 - 2c_4,
\]
and observe that the first two regions above are
\begin{itemize}
	\item[\textup{(I)}] $0 \geq \max\{A,B\}$, and
	\item[\textup{(II)}] $2A^{+} \leq 3c_2 B$.
\end{itemize}

The spectral character of each region is made clear by writing $B$ directly in terms of the eigenvalues. Since $c_k = \sum_i \lambda_i^k$, one computes
\[
B = \biggl(\sum_i \lambda_i^2\biggr)^{\!2} \!- 2\sum_i \lambda_i^4
= 2\sum_{i<j} \lambda_i^2\lambda_j^2 - \sum_i \lambda_i^4.
\]
Thus, $B \leq 0$ if and only if $\sum_i \lambda_i^4 \geq 2\sum_{i<j}\lambda_i^2\lambda_j^2$. This is precisely the case when the distribution of squared eigenvalues $(\lambda_i^2)$ is highly concentrated: when a single eigenvalue satisfies $|\lambda_1| \gg |\lambda_i|$ for all $i \geq 2$, both sums are dominated by $\lambda_1^4$, with the cross terms $\lambda_1^2\lambda_i^2$
$(i \geq 2)$ negligible by comparison, so $B \approx -\lambda_1^4 < 0$. Likewise, the same asymptotic substitution gives $A \approx -\lambda_1^6 < 0$. Region~(I) is therefore the regime where the single-eigenvalue approximation $c_k \approx \lambda_1^k$ is accurate, for which the complete graph $K_n$ is a classical example.

Conversely, when the squared eigenvalues are spread across many comparable values the cross terms dominate and $B > 0$. Bipartite graphs are a classical example: since their eigenvalues come in $\pm$-pairs $\{\pm\lambda_1,\ldots,\pm\lambda_m\}$ and all odd closed
walks vanish ($c_3 = 0$), one obtains
\[
B \;=\; 8\sum_{i < j} \lambda_i^2\lambda_j^2 \;\geq\; 0,
\]
with equality only when at most one $\lambda_i$ is nonzero. Hence every bipartite graph with at least two distinct nonzero eigenvalue magnitudes lies in Region~(II), as do odd sparse non-bipartite graphs such as $C_5$, $C_7$, and the Petersen graph, whose
spectra are similarly distributed. It is in this regime---where no single eigenvalue dominates, the spectral gap is harder to bound, and the interplay between the numbers $c_k$ is genuinely nontrivial---that the bounds of \Cref{thm - main} yield the most useful spectral information.

To highlight the fact that most graphs lie in Region~(II), we explicitly consider the number of graphs in each region for $n=3,4,5,6$ and $7$, which are presented in \Cref{tab:graph_regions_counts}. Note that Region (III) corresponds to the third case in \eqref{eq - cases d=6}. Since this region corresponds to the case where the maximum of $\frac{\sqrt{rs}\|G\|_{X,6}}{\|K_{r,s}\|_{X,6}}$ is attained by the Rademacher distribution, as it was the case when $d=4$, it is reasonable to expect that this distribution will continue to produce good, and often optimal, estimates for higher values of $d$. 

\begin{table}[ht]
	\centering
	\begin{tabular}{|c|r|r|r|r|} 
		\hline
		$n$ & $2^{\binom{n}{2}}$ & Region (I) & Region (II) & Region (III) \\ \hline
		3 & 8 & 0 & 8 & 0 \\
		4 & 64 & 1 & 63 & 0 \\
		5 & 1\,024 & 46 & 978 & 0 \\
		6 & 32\,768 & 1\,212 & 31\,436 & 120 \\
		7 & 2\,097\,152 & 39\,033 & 2\,040\,857 & 17\,262\\
		\hline
	\end{tabular} 
	\vspace{5pt}
	\caption{Number of labeled simple graphs in each region for $n\leq 7$.}
	\vspace{-7pt}
	\label{tab:graph_regions_counts}
\end{table}

\begin{remark}
	A similar reasoning leads us to consider the Schatten $d$-norms as the natural candidate to replace the minimum of $\frac{\|G\|_{X,d}}{\|G_2\|_{X,d}}$ for higher values of $d$. This case was already considered in \Cref{sec - Schatten}, where we obtained one seemingly new inequality, and new proofs of several classical bounds.
\end{remark}

\subsection{Rademacher random variables}

Let $X$ be a random variable following a Rademacher distribution. A direct computation reveals that the moment generating function of $X$ is $M_X(t)=\cosh(t) = (e^t+e^{-t})/2$ for $t\in \mathbb{R}$. Indeed,
\begin{align*}
	\mathbb{E}\bigl[e^{tX}\bigr]& = \mathbb{P}(X=1) \, e^{t} + \mathbb{P}(X=-1)\, e^{-t} = \frac{1}{2}\, e^{t} + \frac{1}{2}\, e^{-t}.
\end{align*} If $d\geq 2$ is even and $G$ has order $n$, then the identity \eqref{thm:Moments} ensures that 
\begin{equation}
	\norm{G}_{X,d}^d=[t^d] \prod_{k=1}^n \frac{e^{\lambda_kt}+e^{-\lambda_kt}}{2}.\vspace{-4pt}\label{eq:Rademacher}
\end{equation} 
Moreover, all the odd cumulants of $X$ vanish and, if $j$ is even, then $\kappa_j=\frac{2^j(2^j-1)B_j}{j}$, where $B_j$ denotes the $j$th Bernoulli number \cite{Gould}. The first few even Bernoulli numbers are $B_2=\frac{1}{6}$, $B_4=-\frac{1}{30}$ and $B_6=\frac{1}{42}$. Given a partition ${\boldsymbol{\pi}}=(\pi_1,\pi_2, \ldots, \pi_\ell)\vdash d$, write $p_{\bm{\pi}}=\pi_1\pi_2\cdots \pi_\ell$, $B_{\boldsymbol{\pi}}=B_{\pi_1}B_{\pi_2}\cdots B_{\pi_\ell}$ and $s_{\bm{\pi}} = \prod_{j=1}^\ell (2^{\pi_j}-1)$, which counts the number of ways, for blocks of sizes $\pi_1,\pi_2,\dots,\pi_\ell$, to choose a nonempty subset inside each block. Using these notations, it follows immediately from \eqref{thm:Cumulants} that if $d\geq 2$ is even and $G$ is a graph of order $n$, then
\begin{align*}\label{prop:PowerSumRademacher}
	\norm{G}_{X,d}^d=2^d\sum_{ \boldsymbol{\pi} \vdash \smash{\frac{d}{2}}}\frac{s_{2\bm{\pi}} B_{2\boldsymbol{\pi}}}{p_{2\boldsymbol{\pi}}y_{2\boldsymbol{\pi}}}c_{2\boldsymbol{\pi}}.
\end{align*} 
In particular, we have
\begin{enumerate}
	\item $\norm{G}_{X,4}^4=\frac{3c_2^2-2c_4}{4!},$
	\item $\norm{G}_{X,6}^6=\frac{15c_2^3-30c_2c_4+16c_6}{6!}$,
	\item $\norm{G}_{X,8}^8=\frac{105c_2^4-420c_2^2c_4+448c_2c_6+140c_4^2-272c_8}{8!}$.
\end{enumerate}
We now evaluate the norms of $M_n$, $K_n$ and $K_{r,s}$, which give the equivalence constants in \Cref{thm - main}.

\begin{proposition}\label{prop:SpecialNormsRademacher}
	Suppose $d\geq 2$ is even. The norms $\norm{\cdot}_{X,d}$ satisfy the following:
	\begin{enumerate}
		\item[(a)] $\norm{K_n}_{X,d}^d=\displaystyle \frac{2^{d+1-n}}{d!} \sum_{k=0}^{n-1} \binom{n-1}{k} k^d$,
		\item[(b)] $\norm{M_n}_{X,d}^d= \displaystyle \frac{2^d}{4^{\lfloor n/2\rfloor}d!} \sum_{k=0}^{2\lfloor n/2\rfloor} \binom{2\lfloor n/2\rfloor}{k} \bigl(\lfloor n/2\rfloor-k\bigr)^d$,
		\item[(c)] $\norm{K_{r,s}}_{X,d}^d= \displaystyle \frac{(4rs)^{d/2}}{2d!}$.
	\end{enumerate}
\end{proposition}

\begin{proof}
	Relation \eqref{eq:Rademacher} tells us that $\norm{K_n}_{X,d}^d$ equals the $d$th Taylor coefficient of $f(t)$, where
	\begin{align*}
		f(t)&= \frac{e^{(n-1)t}+e^{-(n-1)t}}{2} \,\bigg( \frac{e^{t}+e^{-t}}{2}\bigg)^{\!n-1} \!= \frac{ e^{(n-1)t}+e^{-(n-1)t}}{2^n}\sum_{k=0}^{n-1}\binom{n-1}{k} e^{(n-1-2k)t}\\
		&= \sum_{k=0}^{n-1}\binom{n-1}{k}\frac{ e^{(2n-2-2k)t}+e^{-2kt}}{2^n} = \sum_{k=0}^{n-1}\binom{n-1}{k}\sum_{j=0}^{\infty} \frac{(2n-2-2k)^j+(-2k)^j}{2^n j!} \,t^j \\
		&=\sum_{j=0}^{\infty} \Bigg( \frac{2^{j-n}}{j!}\sum_{k=0}^{n-1}\binom{n-1}{k} \Big[(n-1-k)^j+(-k)^j\Big] \Bigg) t^j.
	\end{align*} 
	Observe that
	\begin{equation*}
		\sum_{k=0}^{n-1}\binom{n-1}{k} (n-k-1)^j = \sum_{k=0}^{n-1}\binom{n-1}{k} k^j.
	\end{equation*}
	Therefore, we have
	\begin{align*}
		f(t)=\sum_{j=0}^{\infty} \Bigg( \frac{2^{j-n}}{j!}\sum_{k=0}^{n-1}\binom{n-1}{k} k^j\Big[1+(-1)^j\Big] \Bigg) t^{j},
	\end{align*}
	which implies part (a). Similarly, we also have 
	\begin{align*}
		\Big(\frac{e^{t}+e^{-t}}{2}\Big)^{\!m} = \frac{1}{2^m}\sum_{k=0}^{m}\binom{m}{k} e^{(m-2k)t}= \sum_{j=0}^\infty \Bigg(\frac{1}{2^m j!}\sum_{k=0}^{m}\binom{m}{k} (m-2k)^j \Bigg) t^j.
	\end{align*}
	Taking $m=2\lfloor n/2\rfloor$ in the above and taking the $d$th Taylor coefficient yields (b), and taking $m=2$ while noting the fact that $\norm{K_{r,s}}_{X,d}^d= (rs)^{d/2} \norm{K_{1,1}}_{X,d}^d$ gives (c).
\end{proof}

Using the above values, one may obtain several bounds between $\lambda_1, |\lambda_n|, \|G\|_\ast$ and the random vector norms generated by the Rademacher distribution by exploiting \Cref{thm - main}. In the following, we consider the inequalities obtained in the particular case of $d=6$.

\begin{theorem}\label{thm - Rademacher - d=6}
	Let $G$ be a simple graph of order $n$ with eigenvalues $\bm{\lambda}=(\lambda_1,\dots,\lambda_{n})$. Then, if $r:=\lfloor \frac{n}{2} \rfloor $,\vspace{3pt}
	
	\begin{minipage}{.618\linewidth}
		\begin{enumerate}
			\setlength{\itemindent}{-17pt}
			\item[(a)] \!$\frac{15c_2^3-30c_2c_4+16c_6}{8r(15r^{2}-15r+4)} \leq \lambda_1^6 \leq \frac{(n-1)^5(15c_2^3-30c_2c_4+16c_6)}{n (n^4 + 10 n^3 - 5 n^2 - 70 n + 80)},$
			\item[(b)] \!$ \frac{15c_2^3-30c_2c_4+16c_6}{n (n-1)(n^4 + 10 n^3 - 5 n^2 - 70 n + 80)} \leq \lambda_{n}^6 \leq \frac{15c_2^3-30c_2c_4+16c_6}{32}$,\vspace{6pt}
		\end{enumerate}
	\end{minipage}%
	\begin{minipage}{.37\linewidth}
		\begin{enumerate}
			\setlength{\itemindent}{-12pt}
			\item[(c)] $30c_2^3-60c_2c_4+32c_6 \leq \|G\|_*^6$,\\[-8pt]
			\item[(d)] $\beta^6 \leq \frac{15c_2^3-30c_2c_4+16c_6}{8r(15r^{2}-15r+4)}$,
		\end{enumerate}
	\end{minipage}
	
	\noindent where
	\[
	\beta:= \begin{cases}
		\frac{2}{n} \sum_{j=1}^{n/2} \lambda_j \quad &\text{if } n\text{ is even;}\\
		\frac{2}{n-1} \sum_{j=1}^{(n-1)/2} \lambda_j + \frac{2}{n-1}\min\Bigl\{0,\, \lambda_{\frac{n+1}{2}}\Bigr\} \quad &\text{if } n\text{ is odd.}
	\end{cases}
	\]
	In particular, if $G$ is bipartite, then the inequalities reduce to 
	\[
	\frac{15c_2^3-30c_2c_4+16c_6}{r(15r^{2}-15r+4)} \!\leq\! 8\lambda_1^6 \!\leq\! \frac{15c_2^3-30c_2c_4+16c_6}{4} \!\leq\! \frac{\|G\|_*^6}{8} \!\leq\! \frac{r^5(15c_2^3-30c_2c_4+16c_6)}{15r^{2}-15r+4}.
	\]
\end{theorem}

\section{Other distributions}\label{sec - other}

In the previous sections, we found the best functions to use in \Cref{thm - main} to obtain the sharpest inequalities possible in the case of the random vector norms and their closure whenever $d=2$ or $4$. Moreover, we presented evidence to suggest that the Rademacher distribution and the Schatten $d$-norms are often the best possible for $d\geq 6$. However, due to specific structures that arise in other distributions, it may be possible to obtain sharper bounds even when $d=4$.

In the following, we consider two examples: the standard exponential distribution and the uniform distribution on $[-1,1]$. In the first case, the structure of the expansion in terms of closed walks $c_k$ yields a sharper lower bound for $c_d$. In the second, an inequality of Lata\l a and Oleszkiewicz, specific to this setting, gives a sharp lower bound for $c_6$ whenever $G$ has no triangles or no squares and more than 9 edges.

%In the following, we consider two such examples, namely the standard exponential distribution and the uniform distribution on $[-1,1]$. In the first case, the special structure of the expansion in terms of the number of closed walks $c_k$ yields a better and convenient lower bound for $c_d$. In the second, an inequality of Lata\l a and Oleszkiewicz, specific to this setting, allows us to obtain a sharp lower bound for $c_6$ whenever $G$ has no triangles or no squares and more than 9 edges.

\subsection{Standard exponential random variables}

Suppose that $X$ denotes a standard exponential random variable, which is the random variable with PDF $f_X(x)=e^{-x}$ for $x\geq 0$. The moment generating function of $X$ is given by $M_X(t)=(1-t)^{-1}$ since
\begin{align*}
	\mathbb{E}\bigl[e^{tX}\bigr]=\int_0^{\infty} e^{tx} f_X(x)\,\mathrm{d}x =\int_0^{\infty} e^{-x(1-t)}\,\mathrm{d}x =\frac{1}{1-t}
\end{align*} for $t<1$. Therefore, if $d\geq 2$ is even and $G$ has order $n$, then \eqref{thm:Moments} ensures that
\begin{align}
	\norm{G}_{X,d}^d=[t^d]\prod_{k=1}^n \frac{1}{1-\lambda_kt}.\label{eq:CHSnorms}
\end{align} 
Recall that the \emph{complete homogeneous symmetric} (CHS) polynomial of degree $d$ in the variables $x_1, \ldots, x_n$ is defined as the sum
\begin{align*}
	h_d(x_1, x_2, \ldots, x_n)=\sum_{1\leq i_1\leq i_2\leq\cdots \leq i_d\leq n} x_{i_1}x_{i_2}\cdots x_{i_d}
\end{align*}
of all degree-$d$ monomials in $x_1,  \ldots, x_n$. The CHS polynomials are known to admit the following generating function \cite[Equation 7.11]{Stanley2}:
\begin{align}\label{eq:CHS}
	\sum_{d=0}^{\infty}h_d(x_1, x_2, \ldots, x_n)t^d=\prod_{i=1}^n \frac{1}{1-x_it}.
\end{align} 
Relations \eqref{eq:CHSnorms} and \eqref{eq:CHS} imply that $\norm{G}_{X,d}^d = h_d\bigl(\boldsymbol{\lambda}(G)\bigr)$. Hence, we refer to the norms generated by the standard exponential distribution as \emph{CHS norms}. Hunter was the first to show the positive definiteness of the even degree CHS polynomials on $\mathbb{R}^n$ \cite{Hunter}. Several other proofs have since been put forth \cite{Barvinok, Baston, Bottcher, GarciaOmar, Roventa, Tao}.

Since $X$ has cumulants $\kappa_j=(j-1)!$, the identity \eqref{thm:Cumulants} yields, for even $d\ge 2$,
\begin{align}
	\norm{G}_{X,d}^d
	= h_d(\lambda_1,\lambda_2,\dots,\lambda_n)
	= \sum_{\boldsymbol{\pi}\vdash d} z_{\boldsymbol{\pi}}^{-1} c_{\boldsymbol{\pi}}, \label{eq:Key}
\end{align}
in which $z_{\bm \pi} = \prod_{i\geq 1} i^{m_i} m_i!$. It is a standard fact that $z_{\boldsymbol{\pi}}$ is the size of the centralizer of a permutation in the symmetric group of conjugacy class $\boldsymbol{\pi}$ \cite[Proposition~7.7.3]{Stanley2}. In particular, if $G$ has $m$ edges and $t$ triangles, then
\[
\norm{G}_{X,6}^6=\frac16 c_6+\frac14 m c_4+2t^2+\frac16 m^3.
\]

\medskip
The CHS norms of $K_n$, $K_{r,s}$ and $M_n$ have simple expressions. We leave the evaluation of $\norm{P_n}_{X,d}$ as an open problem (see \Cref{rem:Path}). The following proposition appears in \cite{ChavezExtremal} and follows from \eqref{eq:CHSnorms}.

\begin{proposition}\label{prop:SpecialNormsExp}
	Suppose $d\geq 2$ is even. The CHS norms satisfy the following identities:
	\begin{enumerate}
		\item[(a)] $\norm{K_n}_{X,d}^d=\displaystyle \sum_{k=0}^d (-1)^k(n-1)^{d-k}{k+n-2\choose n-2},$
		\item[(b)] $\displaystyle\norm{M_n}_{X,d}^d= \binom{\lfloor n/2 \rfloor + d/2-1}{d/2}$,
		\item[(c)] $\norm{K_{r,s}}_{X,d}=\sqrt{rs}\,\norm{(1,0, \ldots, 0, -1)}_{X,d}=\sqrt{rs}$.
	\end{enumerate}
\end{proposition}

Interestingly, CHS norms exhibit the same extremal behavior as the spectral norm. Indeed, if $G$ is a connected graph of order $n$, then
\begin{equation}
	2\cos\bigl( \tfrac{\pi}{n+1}\bigr)=\norm{P_n} \le \norm{G} \le \norm{K_n}=n-1 .
	\label{eq:Graph}
\end{equation}
If $T$ is a tree of order $n$, then the sharper bound
\begin{equation}
	2\cos\bigl( \tfrac{\pi}{n+1}\bigr)=\norm{P_n} \le \norm{T} \le \norm{S_n}=\sqrt{n-1}
	\label{eq:Tree}
\end{equation}
holds. Likewise, if $B$ is a connected bipartite graph of order $n$, then
\begin{equation}
	2\cos\bigl( \tfrac{\pi}{n+1}\bigr)=\norm{P_n} \le \norm{B} \le \norm{K_{\lfloor n/2\rfloor,\lceil n/2\rceil}}
	= \sqrt{\big\lfloor \tfrac{n^{\smash{2}}}{4}\big\rfloor } .
	\label{eq:Bipartite}
\end{equation}
The bounds in \eqref{eq:Graph} and \eqref{eq:Tree} are due to Collatz and Sinogowitz \cite{Collatz}, and the upper bound in \eqref{eq:Bipartite} to Bhattacharya, Friedland, and Peled \cite{Bhattacharya}. It is shown in \cite{ChavezExtremal} that CHS norms satisfy the same bounds for trees and connected graphs. We may complete the picture by proving the analogue for connected bipartite graphs.

\begin{theorem}\label{thm:MainTheorem2}
	If $d\geq 2$ is even and $B$ is a connected bipartite graph of order $n$, then 
	\begin{align*}
		\norm{P_n}_{X,d}\leq \|B\|_{X,d} \leq \|K_{\lfloor \frac{n}{2} \rfloor, \lceil \frac{n}{2} \rceil}\|_{X,d}
	\end{align*}
\end{theorem}
\begin{proof}
	The coefficients $z_{\boldsymbol{\pi}}^{-1}$ are positive. \Cref{eq:Key} therefore turns the problem of finding extremal graphs for CHS norms into finding extremal graphs for $c_k$. If $B$ is a bipartite graph with parts of size $r$ and $s$, then $B$ can be embedded inside $K_{r,s}$ as a subgraph, which implies $c_k(B) \leq c_k(K_{r,s})$. Therefore, the relation \eqref{eq:Key} implies the CHS $d$-norm of a connected bipartite graph $B$ is maximized by the complete bipartite graph $K_{r,n-r}$ for some value of $r$. \Cref{prop:SpecialNormsExp} ensures that $\|K_{r,n-r}\|_d = \sqrt{r(n-r)}$, which is maximized when $r=\lfloor\frac{n}{2} \rfloor$. This establishes our upper bound. The lower bound follows from the fact that $\norm{P_n}_{X,d}\leq \norm{T}_{X,d}$ for all trees $T$ of order $n$, which is established in \cite{ChavezExtremal}.
\end{proof}

\begin{remark}\label{rem:Path}
	While we have not been able to prove this fact, we have found that the following appears to hold for all $n\geq 1$ and all even $d\geq 2$:
	\begin{align*}
		\|P_n\|_{X,d}^d &= \frac{2^{n+d}}{n+1}\sum_{k=1}^{n} (-1)^{k+1} \cos^{n-1+d}\!\left(\frac{\pi k}{n+1}\right)\! \sin^{2}\!\left(\frac{\pi k}{n+1}\right) \\
		&= \sum_{j=-\infty}^{\infty}
		\left(\binom{n+d-1}{(n+1)j+\frac{d}{2}}-\binom{n+d-1}{(n+1)j+\frac{d}{2}-1}\right).
	\end{align*}
	The second expression reduces to a finite sum since almost all of the summands vanish. Indeed, the bounds can be replaced by $-\lfloor \frac{d-2}{2\left(n+1\right)} \rfloor$ and $\smash{1+\lfloor\frac{d}{2\left(n+1\right)}\rfloor}$. In particular, if $n+1 \geq \smash{\frac{d}{2}}$, the sum becomes
	\begin{align*}
		\|P_n\|_{X,d}^d = \frac{n^{2}+n-d}{(n+d)(n+d+1)} \binom{n+d+1}{d/2}.
	\end{align*}
\end{remark}

Because of the particular structure of the CHS norms, we are able to establish a sharp lower bound for the number of closed walks of even length in any bipartite graph which often improves on the bound $c_d \geq \frac{\lambda_1^d}{1+n^{1-d}}$ obtained in the case of the Schatten norms. We begin with a computational lemma about the sum of reciprocals of centralizer sizes.

\begin{lemma}\label{lem:EvenParts}
	Let $d\geq 2$ be even. Then, for all real numbers $\alpha$,
	\begin{align*}
		\sum_{\boldsymbol{\pi}\vdash \frac{d}{2}} z_{2\boldsymbol{\pi}}^{-1}\alpha^{\ell(2\boldsymbol{\pi})}=\frac{1}{(\frac{d}{2})!} \Big( \frac{\alpha}{2} \Big)^{\! (\frac{d}{2})},
	\end{align*} in which $(x)^{(k)}=x(x+1)\cdots (x+k-1)$ denotes the rising factorial and $\ell(\bm{\pi})$ is the number of elements in the partition $\bm{\pi} \vdash d$.
\end{lemma}

\begin{proof}
	Since $z_{2\boldsymbol{\pi}}=z_{\boldsymbol{\pi}}2^{\ell(\bm{\pi})}$ and $\ell(2\boldsymbol{\pi})=\ell(\boldsymbol{\pi})$, it follows that
	\[
	\sum_{\boldsymbol{\pi}\vdash \frac{d}{2}} z_{2\boldsymbol{\pi}}^{-1}\alpha^{\ell(2\boldsymbol{\pi})} = \sum_{\boldsymbol{\pi}\vdash \frac{d}{2}} z_{\boldsymbol{\pi}}^{-1} \Big(\frac{\alpha}{2}\Big)^{\!\ell(\boldsymbol{\pi})}
	\]
	The proof of \cite[Prop.~2.2 (b)]{StanleyJack} then ensures that, for all real numbers $\beta$
	\begin{align*}
		\sum_{j=0}^{\infty}\Big(\sum_{\boldsymbol{\pi}\vdash j} z_{\boldsymbol{\pi}}^{-1}\beta^{\ell(\boldsymbol{\pi})}p_{\bm{\pi}}\Big) t^j=\prod_{i=1}^n \frac{1}{ (1-x_it)^{\beta}}.
	\end{align*} 
	where $p_{\bm{\pi}} = x_{\pi_1}^{m_1} \cdots x_{\pi_\ell}^{m_\ell}$. Set $n=1$, $x_1=1$ and $\beta=\frac{\alpha}{2}$ above to obtain the generating function
	\begin{align*}
		\sum_{j=0}^{\infty}\Big(\sum_{\boldsymbol{\pi}\vdash j} z_{\boldsymbol{\pi}}^{-1}\alpha^{\ell(\boldsymbol{\pi})}\Big) t^j&=\Big( \frac{1}{1-t}\Big)^{\!\frac{\alpha}{2}}=\sum_{j=0}^{\infty} \frac{1}{j!} \Big( \frac{\alpha}{2}\Big)^{\!(j)} t^j.
	\end{align*} 
	Simply compare the coefficients of $t^{d/2}$ above to conclude the proof.
\end{proof}

\begin{theorem}\label{thm:MainTheorem5}
	Suppose $d\geq 2$ is even and $B$ is a bipartite graph of order $n$ with $m$ edges and spectral radius $\lambda_1$. The number of closed walks in $B$ of length $d$ satisfies the following bound, which is attained by any complete bipartite graph:
	\begin{align*}
		c_d\geq \Bigg(\frac{2m}{\lambda_1^2} +d- \frac{d}{(\frac{d}{2})!} \Big( \frac{m}{\lambda_1^2} \Big)^{\! (\frac{d}{2})} \Bigg)\lambda_1^d.
	\end{align*}
\end{theorem}
\begin{proof}
	Observe $c_k\leq \lambda_1^{k-2}(\lambda_1^2+\lambda_2^2+\cdots +\lambda_n^2)=2m \lambda_1^{k-2}$. Equality holds when $k\geq 2$ is even and $G$ is a complete bipartite graph. Therefore, we have
	\begin{align}
		c_{\pi_1}c_{\pi_2}\cdots c_{\pi_{\ell}}\leq 2^{\ell}m^{\ell} \lambda_1^{\pi_1+\pi_2+\cdots +\pi_{\ell}}\lambda_1^{-2\ell}=\Big( \frac{2m}{\lambda_1^2}\Big)^{\!\ell} \lambda_1^d\label{eq:EvenLength1}
	\end{align} 
	for any partition $\boldsymbol{\pi}=(\pi_1, \pi_2, \ldots, \pi_{\ell})$. \Cref{lem:EvenParts} with $\alpha= 2m/\lambda_1^2$ then implies
	\begin{equation}
		z_{(d)}^{-1}\Big( \frac{2m}{\lambda_1^2}\Big) + \sum_{\substack{\boldsymbol{\pi}\vdash d/2 \\ \boldsymbol{\pi}\neq (d/2)}} z_{2\boldsymbol{\pi}}^{-1} \Big( \frac{2m}{\lambda_1^2}\Big)^{\!\ell(2\boldsymbol{\pi})} = \frac{1}{(\frac{d}{2})!} \Big( \frac{m}{\lambda_1^2} \Big)^{\! (\frac{d}{2})}.\vspace{-4pt}\label{eq:EvenLength2}
	\end{equation} 
	Relations \eqref{eq:EvenLength1} and \eqref{eq:EvenLength2} imply
	\begin{align}
		\sum_{\substack{\boldsymbol{\pi}\vdash d/2 \\ \boldsymbol{\pi}\neq (d/2)}} z_{2\boldsymbol{\pi}}^{-1} c_{2\boldsymbol{\pi}} \leq \lambda_1^d \sum_{\substack{\boldsymbol{\pi}\vdash d/2 \\ \boldsymbol{\pi}\neq (d/2)}} z_{2\boldsymbol{\pi}}^{-1} \Big( \frac{2m}{\lambda_1^2}\Big)^{\!\ell(2\boldsymbol{\pi})} = \lambda_1^d \Bigg( \frac{1}{(\frac{d}{2})!} \Big( \frac{m}{\lambda_1^2} \Big)^{\! (\frac{d}{2})}-z_{(d)}^{-1}\Big( \frac{2m}{\lambda_1^2}\Big) \Bigg). \label{eq:EvenLength3}
	\end{align} 
	\Cref{prop - random,prop:SpecialNormsExp} imply $\lambda_1\leq \norm{B}_{X,d}$, with equality for any complete bipartite graph. In particular, 
	\begin{align}
		\lambda_1^d\leq \sum_{\boldsymbol{\pi}\vdash d/2} z_{2\boldsymbol{\pi}}^{-1} c_{2\boldsymbol{\pi}} = z_{(d)}^{-1}c_d+\sum_{\substack{\boldsymbol{\pi}\vdash d/2 \\ \boldsymbol{\pi}\neq (d/2)}} z_{2\boldsymbol{\pi}}^{-1} c_{2\boldsymbol{\pi}}.\label{eq:EvenLength4}
	\end{align} 
	Note that the summation for $\norm{B}_{X,d}$ appearing in equation \eqref{eq:Key} is taken over all partitions of $d$ with even parts since bipartite graphs contain no closed walks of odd length. Hence, combining \eqref{eq:EvenLength3} and \eqref{eq:EvenLength4} yield
	\begin{align*}
		\lambda_1^d \leq z_{(d)}^{-1} c_d+\lambda_1^d \Bigg( \frac{1}{(\frac{d}{2})!} \Big( \frac{m}{\lambda_1^2} \Big)^{\! (\frac{d}{2})}-z_{(d)}^{-1}\Big( \frac{2m}{\lambda_1^2}\Big) \Bigg).
	\end{align*} The claim follows after observing $z_{(d)}=d$ and then solving for $c_d$ above.
\end{proof}

%The bound obtained in \Cref{thm:MainTheorem5} is better than the earlier bound of $c_d \geq \frac{\lambda_1^d}{1+n^{1-d}}$ whenever $n\leq 2\bigl((d/2)! (d-1)/d\bigr)^{2/d} \approx \frac{d+\ln(d)+\ln(\pi)}{e}$. Indeed, \Cref{cor - Schatten} ensures that, for bipartite graphs, $\frac{m}{\lambda_1^2} \leq \lfloor \frac{n}{2} \rfloor \leq \frac{n}{2}$. Consequently, if $n\leq 2\bigl((d/2)! (d-1)/d\bigr)^{2/d} $, then
%\begin{align*}
%\frac{2m}{\lambda_1^2} +d- \frac{d}{(\frac{d}{2})!} \Big( \frac{m}{\lambda_1^2} \Big)^{\! (\frac{d}{2})} &\geq %d- \frac{d}{(\frac{d}{2})!} \Big( \frac{m}{\lambda_1^2} \Big)^{\! \frac{d}{2}} \geq 
%d- \frac{d}{(\frac{d}{2})!} \Big( \frac{n}{2} \Big)^{\! \frac{d}{2}} \geq d- \frac{d}{(\frac{d}{2})!} \,\frac{(\frac{d}{2})!(d-1)}{d} =1 \geq \frac{1}{1+n^{1-d}},
%\end{align*}
%where we used the fact that $x^{(d)} \geq x^d$. Hence, \Cref{thm:MainTheorem5} is an improvement over the previously presented bound in this particular case.

\subsection{Uniform random variables}

Now suppose that $X$ is a uniformly distributed random variable on $[-1,1]$, which is the random variable with PDF $f_X(x)=\frac{1}{2}$ for $x\in [-1,1]$. The moment generating function of $X$ is $M_X(t)=\frac{1}{\smash{2t}}(e^t-e^{-t})$ for $t\neq 0$. Indeed,
\begin{align*}
	\mathbb{E}\bigl[e^{tX}\bigr]&=\int_{-1}^1 e^{tx}f_X(x)\,\mathrm{d}x =\frac{1}{2}\int_{-1}^1e^{tx}\,\mathrm{d}x =\frac{1}{2t}(e^t-e^{-t})
\end{align*} 
for $t\neq 0$. If $d\geq 2$ is even and $G$ has order $n$, then the identity \eqref{thm:Moments} ensures that 
\begin{align}
	\norm{G}_{X,d}^d=[t^d] \prod_{k=1}^n \Big(\frac{e^{\lambda_kt}-e^{-\lambda_kt}}{2\lambda_kt}\Big).\label{eq:UniformNorms}
\end{align} 
All the odd cumulants of $X$ vanish. If $j$ is even, then $\kappa_j=\frac{2^{\smash{j}}B_j}{j}$, in which $B_j$ denotes the $j$th Bernoulli number \cite{Gould}. Therefore, \eqref{thm:Cumulants} implies that if $d\geq 2$ is even and $G$ is a graph of order $n$, then
\begin{align*}
	\norm{G}_{X,d}^d = 2^d\sum_{ \boldsymbol{\pi}\vdash \frac{d}{2}}\frac{B_{2\boldsymbol{\pi}} c_{2\boldsymbol{\pi}}}{p_{2\boldsymbol{\pi}}y_{2\boldsymbol{\pi}}},
\end{align*} 
in which $B_{\boldsymbol{\pi}}=B_{\pi_1}B_{\pi_2}\cdots B_{\pi_\ell}$ and $p_{\boldsymbol{\pi}}=\pi_1\pi_2\cdots \pi_\ell$. In particular, if $G$ has $m$ edges, then
\begin{enumerate}
	%\item $\norm{G}_{X,2}^2=\frac{1}{3}m$,
	\item $\norm{G}_{X,4}^4=\frac{1}{18}m^2-\frac{1}{180}c_4,$
	\item $\norm{G}_{X,6}^6=\frac{1}{2835}c_6-\frac{1}{540}mc_4+\frac{1}{162}m^3$.
\end{enumerate}
Once again, the norms of $M_n$, $K_n$ and $K_{r,s}$ admit nice expressions.

\begin{proposition}\label{prop:SpecialNormsUniform}
	Suppose $d\geq 2$ is even. The norms $\norm{\cdot}_{X,d}$ satisfy the following:
	\begin{enumerate}
		\item[(a)] $\norm{K_n}_{X,d}^d=\displaystyle \frac{2^{d+1}(n-2)!}{(n+d)!} \sum_{k=0}^{n-1}\frac{(-1)^{n-1-k} k^{n+d}}{(n-1-k)!k!} = \frac{2^{d+1}(n-2)!}{(n+d)!}\, S(n+d,n-1)$,
		\item[(b)] $\displaystyle \|M_n\|_{X,d}^d = \frac{2^{d}}{(d+2\lfloor n/2\rfloor)!}\sum_{k=0}^{2\lfloor n/2\rfloor}\binom{2\lfloor n/2\rfloor}{k} (\lfloor n/2\rfloor-k)^{d+2\lfloor n/2\rfloor}$,
		\item[(c)] $\displaystyle \norm{K_{r,s}}_{X,d}^d=\sqrt{rs}\norm{ (1, 0, \ldots, 0, -1)}_{X,d}^d = \frac{2(4rs)^{d/2}}{(d+2)!}$,
	\end{enumerate}
	where the $S(n,k)$ are the Stirling numbers of the second kind.
\end{proposition}

\begin{proof}
	Relation \eqref{eq:UniformNorms} tells us that $\norm{K_n}_{X,d}^d$ equals the $d$th Taylor coefficient of $f(t)$, where
	\begin{align*}
		f(t) &= \frac{e^{(n-1)t}-e^{-(n-1)t}}{2(n-1)t} \,\Big( \frac{e^{t}-e^{-t}}{2t}\Big)^{\!n-1} =\, \frac{e^{(n-1)t}-e^{-(n-1)t}}{2^n(n-1)t^n} \sum_{k=0}^{n-1}\binom{n-1}{k}(-1)^k e^{(n-1-2k)t}\\
		&= \frac{1}{2^n(n-1)t^n}\sum_{k=0}^{n-1}\binom{n-1}{k}(-1)^k\big( e^{(2n-2k-2)t}-e^{(-2k)t} \big)\\
		&= \frac{1}{2^n(n-1)t^n}\sum_{k=0}^{n-1}\binom{n-1}{k}(-1)^k\Bigg(\sum_{j=0}^{\infty} \frac{(2n-2k-2)^j-(-2k)^j}{j!} t^j \Bigg)\\
		&= \frac{1}{n-1}\sum_{j=0}^{\infty} \Bigg( \frac{2^{j-n}}{j!}\sum_{k=0}^{n-1}\binom{n-1}{k}(-1)^k \Big[(n-k-1)^j-(-k)^j\Big] \Bigg) t^{j-n}.
	\end{align*} 
	Observe that
	\begin{align*}
		\sum_{k=0}^{n-1}\binom{n-1}{k}(-1)^k (n-k-1)^j = (-1)^{n-1}\sum_{k=0}^{n-1}\binom{n-1}{k}(-1)^k k^j.
	\end{align*}
	Therefore, we have
	\begin{align*}
		f(t)=\sum_{j=0}^{\infty} \Bigg( \frac{2^{j-n}}{(n-1) j!}\sum_{k=0}^{n-1}\binom{n-1}{k}(-1)^k k^j\Big[(-1)^{n-1}-(-1)^j\Big] \Bigg) t^{j-n},
	\end{align*}which implies part (a). 
	Similarly, we also have 
	\begin{align*}
		\Big(\frac{e^{t}+e^{-t}}{2t}\Big)^{\!m} = \frac{1}{2^mt^m}\sum_{k=0}^{m}\binom{m}{k} e^{(m-2k)t}= \sum_{j=0}^\infty \Bigg(\frac{1}{2^m j!}\sum_{k=0}^{m}\binom{m}{k} (m-2k)^j \Bigg) t^{j-m}.
	\end{align*}
	Taking $m=2\lfloor n/2\rfloor$ in the above and considering the $d$th coefficient yields (b), and taking $m=2$ while noting the fact that $\norm{K_{r,s}}_{X,d}^d= (rs)^{d/2} \norm{K_{1,1}}_{X,d}^d$ gives (c).
\end{proof}

Suppose $X_1, X_2, \ldots, X_n$ are independent random variables which are uniformly distributed on $[-1,1]$. Furthermore, assume that $\boldsymbol{\lambda}=(\lambda_1, \lambda_2, \ldots, \lambda_n)$ and $\boldsymbol{\mu}=(\mu_1, \mu_2, \ldots, \mu_n)$ are vectors for which\vspace{-4pt} 
\begin{align*}
	(\lambda_1^2, \lambda_2^2, \ldots, \lambda_n^2)\prec (\mu_1^2, \mu_2^2, \ldots, \mu_n^2).
\end{align*} 
Lata\l a and Oleszkiewicz showed in \cite{Latala} that $\mathbb{E}|\sum_{k=1}^n\lambda_kX_k|^d\geq \mathbb{E}|\sum_{k=1}^n\mu_kX_k|^d$ for $d\geq 2$. This result allows us to bound $\norm{G}_{X,d}$ from below by $\sqrt{m}$ in the case $G$ contains either no triangle or no square.

Mantel famously proved that the number of edges in a triangle-free graph of order $n$ satisfies the bound $m\leq \lfloor n^2/4\rfloor$. A spectral version of this result is due to Nosal, who proved in \cite{Nosal} that the spectral radius of any triangle-free graph satisfies $\lambda_1\leq \sqrt{m}$. Nikiforov further showed in \cite{NikiforovSquares} that the spectral radius of any square-free graph of order $n$ and $m\geq 9$ edges satisfies $\lambda_1\leq \sqrt{m}$ with equality attained by $S_n = K_{1,n-1}$. Using these particular bounds, we may obtain the following theorem.

\begin{theorem}\label{thm:Khintchine1}
	Let $G$ be a simple graph of order $n$ with $m$ edges and $X\sim \mathcal{U}(-1,1)$. If $G$ either (a) has no triangle or (b) has no square and $m\geq 9$, then $\norm{G}_{X,d}$ satisfies the bound
	\begin{align*}
		\norm{G}_{X,d}^d\geq \frac{2(2\sqrt{m})^{d}}{(d+2)!}.
	\end{align*} 
	Moreover, equality is realized in (a) by any complete bipartite graph and in (b) by $S_n$.
\end{theorem}

\begin{proof}
	Let $\boldsymbol{\lambda}=(\lambda_1,\dots,\lambda_n)$ be the eigenvalues of $G$, and set $\boldsymbol{\nu}=(\sqrt{m},0,\ldots,0,-\sqrt{m}).$ Note that $\sum_{k=1}^n \lambda_k^2 = 2m = \sum_{k=1}^n \nu_k^2.$
	Since $\nu_1=\sqrt{m}$, the bound $\lambda_1\le \sqrt{m}$ yields $\lambda_1\le \nu_1$ in cases~(a) and~(b), and therefore, $\lambda_1^2\le \nu_1^2$. Moreover, for every $2\le j\le n$,
	\[
	\sum_{k=1}^j \lambda_k^2 \le 2m = \sum_{k=1}^j \bigl(\nu_k^2\bigr)^{\downarrow}.
	\]
	It follows that $(\lambda_1^2,\dots,\lambda_n^2)$ is majorized by $(\nu_1^2,\dots,\nu_n^2)$. While $\boldsymbol{\nu}$ need not be the spectrum of a graph of order $n$, its nonzero entries are precisely those of $K_{1,m}$, and zeros do not affect the corresponding random vector norm. Thus, the bound of Lata\l a and Oleszkiewicz and \Cref{prop:SpecialNormsUniform} yield\vspace{-1.5pt}
	\begin{equation*}
		\norm{G}_{X,d}^d=\frac{1}{d!}\,\mathbb{E}\left|\sum_{k=1}^n\lambda_kX_k\right|^d \!\geq \frac{1}{d!}\,\mathbb{E}\left|\sum_{k=1}^n\nu_kX_k\right|^d \!=  \frac{2(2\sqrt{m})^{d}}{(d+2)!},\vspace{-1.5pt}
	\end{equation*}
	as desired.
\end{proof}

\begin{corollary}\label{cor:Khintchine1}
	Let $G$ be a simple graph of order $n$ with $m$ edges. If $G$ has either (a) no triangle or (b) no square and $m\geq 9$, then 
	\begin{align*}
		4 c_6+34m^3\geq 21mc_4.
	\end{align*} 
	Moreover, equality is realized in (a) by any complete bipartite graph and in (b) by $S_n$.
\end{corollary}

\begin{proof}
	Set $d=6$ in \Cref{thm:Khintchine1} to obtain $\frac{1}{2835}c_6-\frac{1}{540}mc_4+\frac{1}{162}m^3\geq \frac{2(2\sqrt{m})^{6}}{(6+2)!}$. Simplifying yields the desired inequality.
\end{proof}

Classically, one would bound $c_6$ by applying the Hölder inequality to the eigenvalue vector, yielding
\[
c_4^2 \leq c_2c_6 = 2mc_6,
\]
that is $c_6\geq c_4^2/2m$. One may easily verify that the bound in \Cref{cor:Khintchine1} is an improvement on the former precisely when $c_4 \geq 2m^2$. Indeed, we have
\[
c_4 = \sum_{i=1}^n \lambda_i^4 \leq \left(\sum_{i=1}^n \lambda_i^2 \right)^{\!\!2} = c_2^2=4m^2,
\]
which implies that $17m^2-2c_4 \geq 9m^2\geq0$  (and also that $c_4\geq 2m^2$ is not an empty condition). Therefore,
\[
\frac{21mc_4-34m^3}{4} = \frac{c_4^2}{2m}+ \frac{(17m^2-2c_4)(c_4-2m^2)}{4m} \geq \frac{c_4^2}{2m},
\]
proving the claim.

As a final observation, we establish a lower bound for the number of closed walks of length $6$ in a triangle-free graph. To this end, the Cauchy--Schwarz inequality implies 
\begin{align}
	Z_1\geq \frac{4m^2}{n}\label{eq:ZagrebCauchy}
\end{align} with equality if and only if $G$ is regular \cite{deCaen,Ilic,Yoon}. 

\begin{corollary}\label{cor:Khintchine2}
	Let $G$ be a triangle-free graph of order $n$ with $m$ edges and $q$ squares. Then
	\begin{align*}
		%c_6\geq 21q-\frac{21}{4}m+\Big(\frac{21}{n}-\frac{17}{2}\Big)m^2.
		c_6 \geq \frac{21m( 4q-m )}{2}-\left(\frac{17}{2}-\frac{42}{n}\right)\!m^3
	\end{align*} 
	Moreover, equality is realized when $n$ is even by the complete bipartite graph with equal parts.
\end{corollary}

\begin{proof}
	The identity $c_4=8q-2m+2Z_1$, together with \Cref{cor:Khintchine1} and \eqref{eq:ZagrebCauchy}, implies 
	\begin{equation*}
		4c_6+34m^3 \geq 21m( 8q-2m+2Z_1 ) \geq  42m\!\left( 4q-m+\frac{4m^2}{n} \right),
	\end{equation*} 
	in which equality holds for any complete bipartite graph. Solving for $c_6$ completes the proof. Note that any regular complete bipartite graph must be of the form $K_{r,r}$, which establishes sharpness.
\end{proof}

\section{Concluding remarks}

\Cref{thm - main} provides a general method for deriving sharp inequalities on simple graphs from majorization and Schur-convexity. In this way, the problem is reduced to the choice of a positive Schur-convex spectral functional. Random vector norms form a convenient class for this purpose since they are Schur-convex and admit explicit expansions in terms of closed walks. In degree $4$, this optimization problem can be solved completely within that class.

The present work leaves several natural questions open.

\begin{enumerate}
	\item[(i)] Determine the optimal functionals in \Cref{thm - main} for higher values of $d$.
	\item[(ii)] Give a structural description of the optimization problem in degree $6$, where different functionals become optimal in different regions.
	\item[(iii)] Decide whether similar decompositions occur for larger even values of $d$, and, if so, whether these regions admit a direct spectral or combinatorial interpretation.
	\item[(iv)] Clarify whether the Rademacher and Schatten norms continue to play an important extremal role in higher degree.
	\item[(v)] Find further positive Schur-convex spectral functionals for which the framework produces explicit sharp inequalities.
	\item[(vi)] Find ways, if possible, to use this technique on connected graphs to get sharper results.
\end{enumerate}

%%%%%%%%%%%%%%%%%%%%%%%%%%%%%%%%%%%%%%%%%%%%%%%%%%%%%%%%%%%%%%%%%%%%%%%%%%%
%%%%%%%%%%%%%%%%%%%%%%%%%%%%%%%%%%%%%%%%%%%%%%%%%%%%%%%%%%%%%%%%%%%%%%%%%%%
%%%%%%%%%%%%%%%%%%%%%%%%%%%%%%%%%%%%%%%%%%%%%%%%%%%%%%%%%%%%%%%%%%%%%%%%%%%
%%%%%%%%%%%%%%%%%%%%%%%%%%%%%%%%%%%%%%%%%%%%%%%%%%%%%%%%%%%%%%%%%%%%%%%%%%%
%%%%%%%%%%%%%%%%%%%%%%%%%%%%%%%%%%%%%%%%%%%%%%%%%%%%%%%%%%%%%%%%%%%%%%%%%%%

\bibliographystyle{plain}
\bibliography{Graphs3.bib}

\end{document}